\documentclass{amsart}%
\usepackage{amsfonts}
\usepackage{amsmath}
\usepackage{amssymb}
\usepackage{graphicx}%
\newtheorem{theorem}{Theorem}
\theoremstyle{plain}

\newtheorem{cor}{Corollary}

\newtheorem{lemma}{Lemma}

\newtheorem{proposition}{Proposition}

\theoremstyle{definition}
\newtheorem{definition}{Definition}
\newtheorem{example}{Example}
\newtheorem{remark}{Remark}
\newcommand{\R}{{\mathbb R}}

\newcommand{\oo}{\overleftarrow}

\newcommand{\ot}{{\overleftarrow{\theta}}}
\numberwithin{equation}{section}

\begin{document}
\title{Tilings from Graph Directed Iterated Function Systems}
\author{Michael Barnsley and Andrew Vince}
\address{Mathematical Sciences Institute, Australian National University, Canberra,
ACT, Australia \\
michael.barnsley@anu.edu.au}
\address {Department of Mathematics, University of Florida, Gainesville, FL, USA \\
avince@ufl.edu}
\date{}

\keywords{tiling, graph directed iterated function system, self-similar}

\subjclass[2020]{52C20, 52C22, 05B45, 28A80}

\begin{abstract}
A new method for constructing self-referential tilings of Euclidean space from
a graph directed iterated function system, based on a combinatorial structure
we call a pre-tree, is introduced. In the special case that we refer to as
balanced, the resulting tilings have a finite set of prototiles, are
quasiperiodic but not periodic, and are self-similar. A necessary and
sufficient condition for two balanced tilings to be congruent is provided.

\end{abstract}

\maketitle

\section{Introduction}

Iterated function systems (IFSs) have been at the heart of fractal geometry
almost from its origins, and the attractor $A$ of the IFS has played the
central role in the theory. The attractor is the union of contracted images of
itself. In the case that the contractions are similarities, the attractor is
the union of smaller similar copies of itself. See \cite{hutchinson} for formal background on iterated function systems
(IFS).   Here we are concerned with a generalization of an IFS,  called a \textit{graph directed iterated function system} (GIFS).  A GIFS
generalizes the IFS concept so that the attractor has components $A_{1}, A_{2},\dots,A_{n}$, each component $A_{i}$ the union of contracted copies of
the components $A_{1},A_{2},\dots,A_{n}$. The prescription for how this is
done is encoded in a directed graph (with possibly multiple edges and loops),
the precise definition of a GIFS given in Section~\ref{sec:def1}. An IFS is the special case of a GIFS
in which the digraph has a single vertex.   Early work related to GIFSs  includes
\cite{ bandtTILE, barnsleyFE2,bedfordGraph, das, dekking, MW, werner}. In
some of these works, a  graph IFS are referred to as a recurrent IFS.

The paper \cite{BV1} introduced a method for constructing certain tilings
using an IFS, in which the tiles are related to the attractor of the IFS.
This construction included, as special cases, digit tilings \cite{V} (see the
right panel of Figure~\ref{fig:chair} for the twin dragon tiling),
crystallographic tilings \cite{G}, certain substitution tilings like the
\textquotedblleft chair tiling" (see the left panel of Figure~\ref{fig:chair}%
), and new tilings (see, for example, Figure~\ref{fig:new}). In \cite{BV2}
this method was engineered to obtain tilings of Euclidean space having the
properties of repetitivity (quasiperiodicity) and self-similarity (whose definitions appear later
in this introduction). In these tilings there are finitely many tiles up to
congruence, and all tiles are similar to the attractor of the IFS (see
Figure~\ref{fig:EX3}).

\begin{figure}[htb]
\includegraphics[width=4.5cm, keepaspectratio]{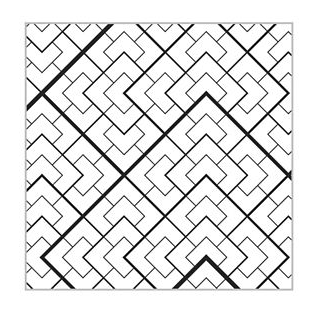}\hskip 1cm \includegraphics[width=4cm, keepaspectratio]{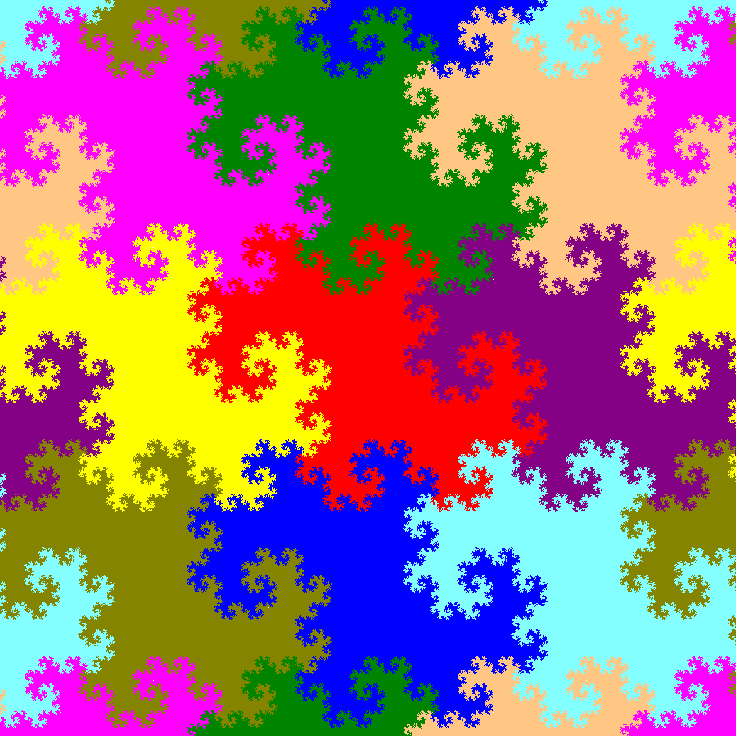}
\caption{the chair tiling and the twin dragon tiling.}
\label{fig:chair}
\end{figure}

\begin{figure}[htb]
\includegraphics[width=9cm, keepaspectratio]{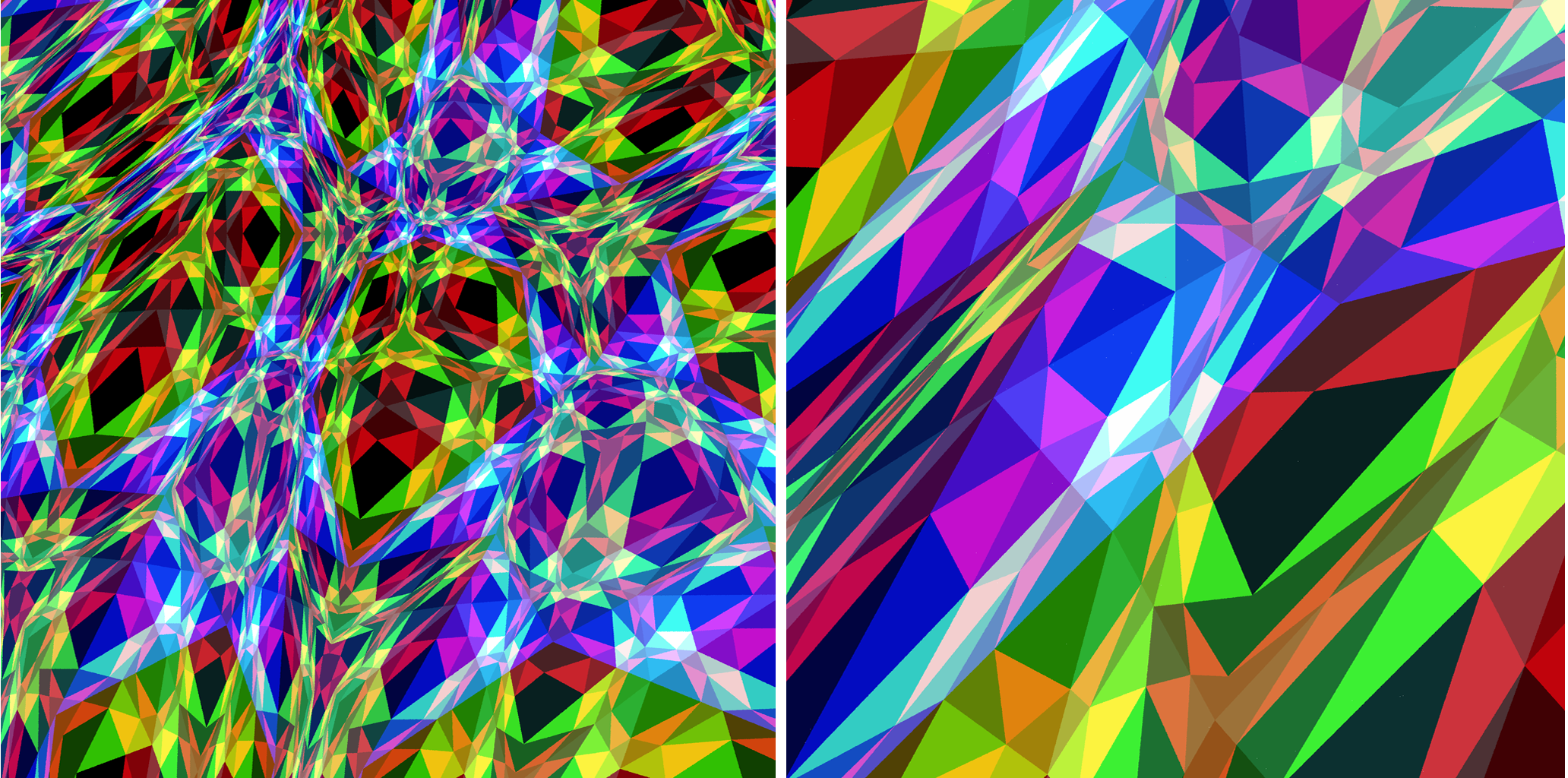}
\caption{Two views of the same IFS tiling with a triangular attractor. All tiles are triangles (the black
quadrilateral is the union of two black triangular tiles).}
\label{fig:new}
\end{figure}

\begin{figure}[htb]
\vskip 3mm
\includegraphics[width=7cm, keepaspectratio]{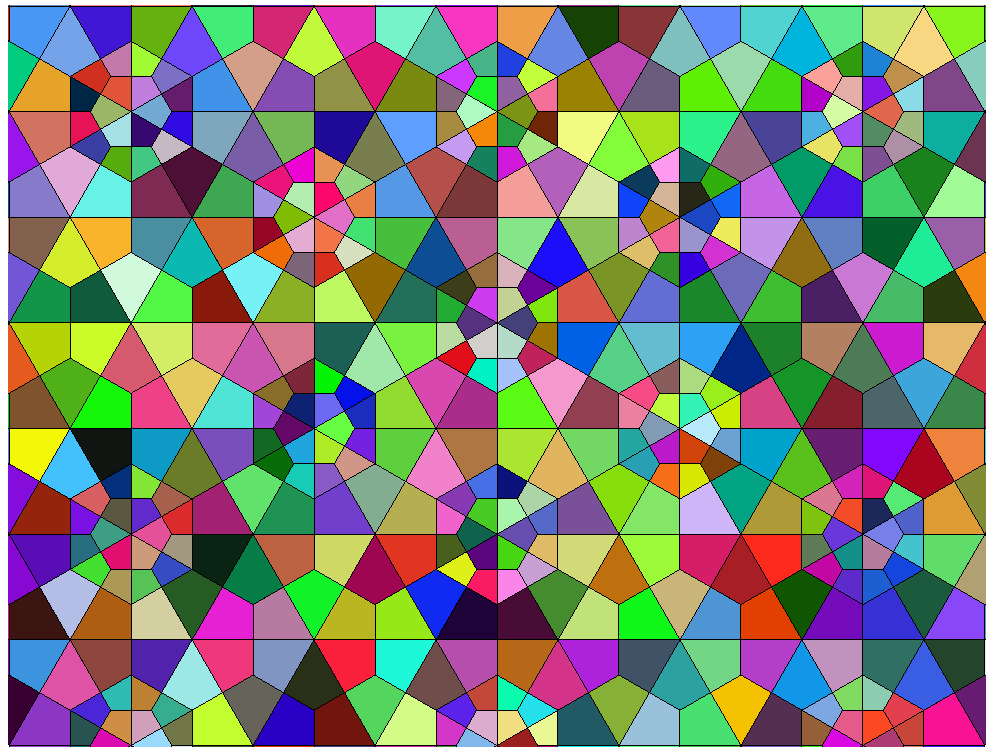}
\caption{A tiling by copies of tiles similar to the attractor of an IFS.}
\label{fig:EX3}
\end{figure}

The goal of this paper is twofold. First, using a GIFS, we introduce a new
unifying method for the construction of tilings of $\R^d$. The method is based on a
combinatorial structure in the underlying directed graph that we call a
\textit{pre-tree}. This notion is defined precisely in Definition~\ref{def:pt} in Section~\ref{sec:pt}. Our point of view is more graph theoretic than is usual in papers on
fractals via iterated function systems; so for basic notions about graphs see \cite{D}. The approach in this paper is a
substantial extension of some concepts in \cite{BV1} and \cite{BV2}.  Other papers which relate to the use of GIFS to construct tilings include
\cite{bandtTILE,mcclure1,mcclure2}.  

A given GIFS has infinitely many \textit{parameters}, the term parameter
defined in Section~\ref{sec:ps}, and for each parameter
${\overleftarrow{\theta}}$, many possible ${\overleftarrow{\theta}}%
$-\textit{sequences} $\mathcal{S}$ (see Definition~\ref{def:ts} of ${\overleftarrow{\theta}}
$-sequences in Section~\ref{sec:ts}). The main result in the first part of the
paper is that, for a given GIFS and for each parameter ${\overleftarrow{\theta}}$ and each
corresponding ${\overleftarrow{\theta}}$-sequence $\mathcal{S}$, there is a
tiling $T({\overleftarrow{\theta}}, \mathcal{S})$ of Euclidean space (see
Theorem~\ref{thm:tiling} in Section~\ref{sec:tilings}). Addresses can be
assigned in a natural way to the tiles in $T({\overleftarrow{\theta}},
\mathcal{S})$ as described in Section~\ref{sec:a}. In addition to the tiling
examples mentioned in the paragraph above, these GIFS tilings include the
Penrose and Rauzy tilings \cite{Rauzy}, and many others; see, for example,
Figures~\ref{fig:a}, \ref{fig:b} and \ref{fig:uniform}.

\begin{figure}[htb]
\includegraphics[width=6cm, keepaspectratio]{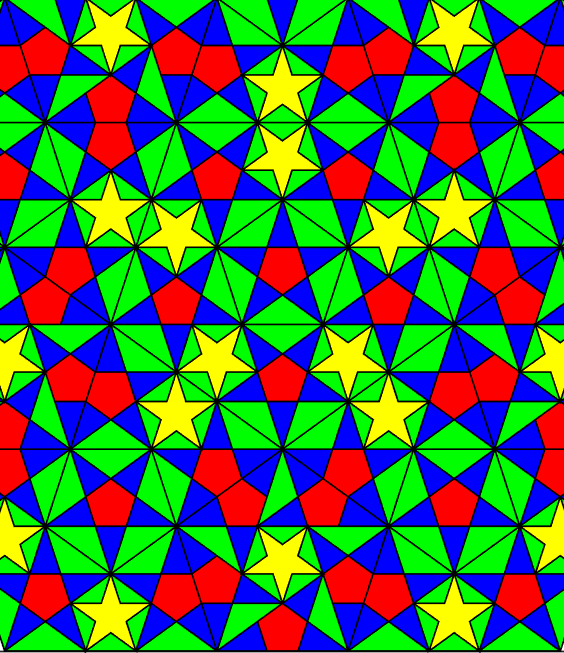}
\caption{A GIFS tiling by scaled copies of four types of tiles.}
\label{fig:a}
\end{figure}

Although the tilings constructed by our general method possess a replicating
pattern, particularly nice properties are obtained by carefully choosing the
${\overleftarrow{\theta}}$-sequences. We call such special
${\overleftarrow{\theta}}$-sequences and the corresponding tilings
\textit{balanced} (see Definition~\ref{def:balanced} in Section~\ref{sec:ts} and Definition~\ref{def:balanced2} in Section~\ref{sec:examples}). 
The second goal of this paper is to use the balanced property to construct
tilings $T$ that satisfy the following properties:

\begin{itemize}
\item $T$ has finitely many tiles up to congruence;
\item $T$ is self-similar;
\item $T$ is repetitive, but not periodic;
\item for a given GIFS, there are uncountably many tilings up to isometry
(each associated with a particular parameter).
\end{itemize}

Note that the tiles do not have to be pairwise similar. The relevant
definitions are:

\begin{definition}
A tiling $T$ is \textbf{repetitive}, also called \textbf{quasiperiodic}, if,
for every finite patch $P$ of $T$, there is a real number $R$ such that every
ball of radius $R$ contains a patch isometric (congruent) to $P$.
\end{definition}

\begin{definition}
A tiling $T$ is \textbf{self-similar} if there exists a similarity
transformation $\phi: {\mathbb{R}}^{d} \rightarrow{\mathbb{R}}^{d}$ with
scaling ratio greater than $1$ such that, for every tile $t\in T$, its image
$\phi(t)$ is, in turn, tiled by tiles in $T$.
\end{definition}

\begin{figure}[htb]
\includegraphics[width=7cm, keepaspectratio]{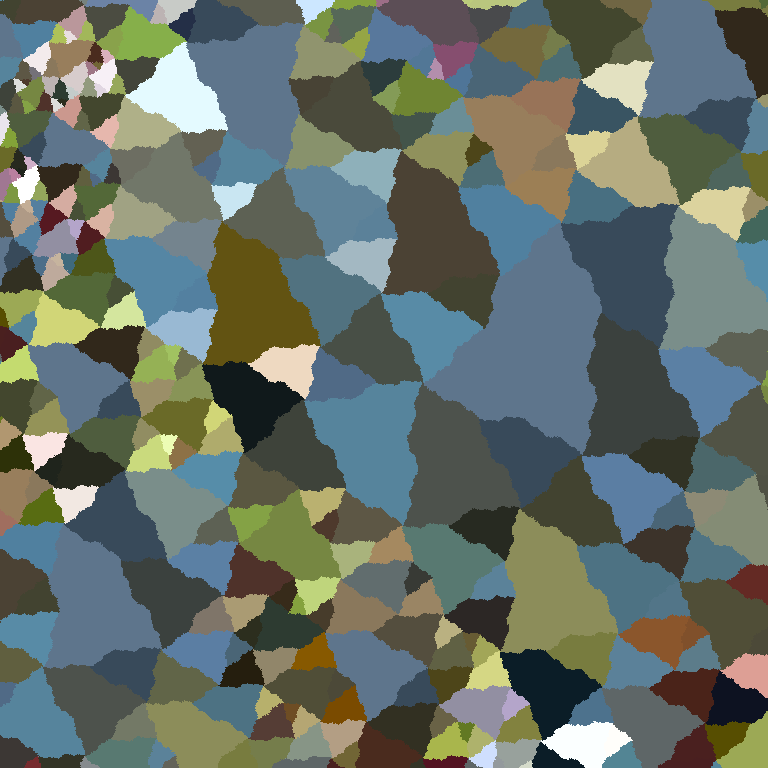}
\caption{A balanced fractal tiling based on a ``triangle" tile due to Akiyama \cite{aki}.}
\label{fig:b}
\end{figure}

Figures~\ref{fig:chair} (left panel), \ref{fig:EX3}, \ref{fig:a},
\ref{fig:b}, and \ref{fig:fp} are examples of such balanced tilings. Note that the tiles used in
Figures~\ref{fig:a}, \ref{fig:b} and  \ref{fig:fp} are not pairwise similar. Work on
self-similar tilings has been extensive over the past few decades, dating back
at least to the Penrose tilings \cite{Pen}, popularized by M. Gardner
\cite{g}. Much recent research was motivated by W. Thurston's notes \cite{T}.
Some formative papers include \cite{b,gh,Ken,rw,s,V}.

Among the main results of the second part of the paper are that the
constructed balanced tilings indeed satisfy the above bulleted properties.
Many of the proofs, which appear in Section~\ref{sec:properties}, rely heavily
on the hierarchical properties of the tilings, as described in
Section~\ref{sec:hier}. The \textit{tiling map} ${\overleftarrow{\theta}}
\mapsto T({\overleftarrow{\theta}})$ is a continuous map from the parameter
space, as defined in Section~\ref{sec:ps}, to the tiling space, as defined in
Section~\ref{sec:PT}. Theorem~\ref{thm:comm} states that the shift map on the
parameter space corresponds via this tiling map to going up one level in the
tiling hierarchy in the tiling space. There is an easily defined equivalence
relation on the parameter space (see Definition~\ref{def:e} in Section~\ref{sec:cong}) such that the
tiling map induces a bijection between the set of parameters up to equivalence
and the set of balanced tilings up to congruence (isometry); see Theorem~\ref{thm:main}  in Section~\ref{sec:cong}.
 This provides a
necessary and sufficient condition for two balanced tilings to be congruent.

\section{Graph Directed Iterated Function Systems}

\label{sec:GIFS}

\subsection{GIFS}

\label{sec:def1} Let $\mathbb{N}$ denote the set of positive integers, and for
$n \in\mathbb{N}$, let $[n] := \{1,2, \dots, n\}$. Let $G =\left(  V,
E\right)  $ be a finite strongly connected directed graph with vertex set $V
=[n]$ and edge set $E$. \textit{Strongly connected} means that, for any two
vertices $i$ and $j$, there is a directed path from $i$ to $j$. The digraph
$G$ may have loops and/or multiple edges. For an edge $e = (i,j)$, directed
from vertex $i$ to vertex $j$, the vertex $j$ is called the \textit{head},
denoted $e^{+}$, and the vertex $i$ is called the \textit{tail}, denoted
$e^{-}$. Let $E_{i}$ denote the set of all edges whose tail is $i$.

In this paper, a path always means a directed path, and a path can have
repeated vertices and edges. An infinite path has a starting vertex but no
terminal vertex. A path $\sigma=\sigma_{1}\, \sigma_{2}\, \cdots$ will
be written as its ordered string of edges $\sigma_{i} \in E, \; i = 1,2 \dots$. The starting vertex of a path $\sigma$ will be denoted $\sigma^{-}$, and the
terminal vertex of a finite path by $\sigma^{+}$. The length of a finite path
$\sigma$, i.e., the number of edges, will be denoted $|\sigma|$. We allow paths of length zero, consisting of a single vertex.

To each directed edge $e \in E$ of the digraph $G$, associate a function
$f_{e}: {\mathbb{R}}^{d}\rightarrow{\mathbb{R}}^{d}$. Let
\[
F = \{f_{e} \, : \, e\in E\}.
\]
Each $f\in F$ is assumed throughout this paper to be an invertible
contraction. Thus there is a contraction factor $\lambda_{f} <1$ such that $|
f(x)-f(y)|\leq\lambda_{f} | x-y |$ for all $f\in F$ and all $x,y\in
\mathbb{R}^{d}$. The pair $(G,F)$ is called a \textbf{graph directed IFS} (GIFS).

Let $\mathbb{H}$ denote the set of nonempty compact subsets of $\mathbb{R}%
^{d}$, and define $\mathbf{F}: \mathbb{H}^{n}\rightarrow\mathbb{H}^{n}$ as
follows. If $\mathbf{X} = (X_{1},X_{2}, \dots, X_{n}) \in{\mathbb{H}}^{n}$,
then
\[
\mathbf{F}(\mathbf{X}) = \big (F_{1}(\mathbf{X}), F_{2}(\mathbf{X}), \dots,
F_{n}(\mathbf{X}) \big ),
\]
where
\[
F_{i}(\mathbf{X}) = \bigcup_{e\in E_{i}} f_{e}(X_{e^{+}}).
\]
Note that, if $G = (V,E)$ and $V$ consists of a single vertex and $E$ is a set
of loops, then a graph directed IFS is an ordinary IFS, and $\mathbf{F}$ is
the associated Hutchinson operator. Thus a GIFS is a generalization of an IFS.

\subsection{Notation}

\label{sec:notation}

Let $\Sigma^{*}$ denote the set of paths of finite length in the digraph $G$
and $\Sigma^{\infty}$ the set of paths of infinite length. For $\sigma=
\sigma_{1} \sigma_{2} \cdots\in\Sigma^{\infty}$ let
\[
\sigma|k = \sigma_{1} \sigma_{2} \cdots\sigma_{k} \in\Sigma^{*},
\]
and $\sigma|0$ the path that is just the vertex $\sigma^{-}$.

Denote by $\overleftarrow{G}$ the digraph obtained from $G$ by reversing the
direction on all edges. Define path spaces $\overleftarrow{\Sigma}^{*}$ and
$\overleftarrow{\Sigma}^{\infty}$ accordingly. For any edge $e$ in $G$, let
$\overleftarrow{e}$ be the oppositely directed edge, and define
\begin{equation}
\label{eq:n1}f_{\overleftarrow{e}} := (f_{e})^{-1}.
\end{equation}
For $\sigma= \sigma_{1} \sigma_{2} \sigma_{3} \cdots\sigma_{k} \in\Sigma^{*}
\cup\overleftarrow{\Sigma}^{*}$, define
\begin{equation}
\label{eq:n2}f_{\sigma} := f_{\sigma_{1}}\circ f_{\sigma_{2}}\circ
f_{\sigma_{3}}\cdots f_{\sigma_{k}}.
\end{equation}
Note that, for $\sigma\in\Sigma^{*}$, the map $f_{\sigma}$ is a contraction,
but for $\sigma\in\overleftarrow{\Sigma}^{*}$, the map $f_{\sigma}$ is an expansion.

\subsection{The Attractor}

\label{sec:attractor}

The following is a standard result in the theory of graph iterated function
systems. In the theorem $\mathbf{F}^{k}$ denotes the $k$-fold iteration of
$\mathbf{F}$.

\begin{theorem}
\label{thm:one} If $\left(  G, F\right)  $ is a GIFS such that each function
in $F$ is a contraction, then there exists a unique $\mathbf{A}=(A_{1}%
,A_{2},...,A_{n} )\in\mathbb{H}^{n}$ such that
\[
\mathbf{A=F(A)}
\]
and
\[
\mathbf{A=}\lim_{k\rightarrow\infty} {\mathbf{F}}^{k}(\mathbf{B})
\]
independent of $\mathbf{B}\in\mathbb{H}^{n},$ where convergence is with
respect to the Hausdorff metric on $\mathbb{H}^{n}$.
\end{theorem}

The above limit formula is equivalent to
\begin{equation}
\label{eq:3}A_{i} =\bigcup \big \{ \lim_{k \rightarrow\infty} \{ f_{\sigma} (B) \, : \,
|\sigma| = k, \, \sigma^{-} = i \big \},
\end{equation}
for all $i \in[N]$ and for any $B \in\mathbb{H}.$ With notation as in Theorem
\ref{thm:one}, $\mathbf{A}$ is called the \textbf{attractor} of the GIFS
$(G,F)$ and $\left\{  A_{i}: i \in[n]\right\}  $ is its set of
\textbf{components}, each of which is compact. It follows from the definition
that, for each $i \in[n]$, the components of the attractor have the property
that
\begin{equation}
\label{eq:decomp}A_{i} = \bigcup_{e\in E_{i}} f_{e}(A_{e^{+}}).
\end{equation}

If, for all $i\in[n]$, the intersection of distinct sets in the union in
Equation~\eqref{eq:decomp} has empty interior, then the attractor $\mathbf{A}$
of the GIFS is said to be \textit{non-overlapping}. \textbf{For all GIFSs in
this paper we assume that the attractor is non-overlapping and that each
component of the attractor has nonempty interior}.

\section{Pre-Trees and $\theta$-Sequences}

\label{sec:CPT}

\subsection{Pre-Tree} \label{sec:pt}

\begin{definition}\label{def:pt} A \textit{subpath} of a finite path $\sigma= \sigma_{1}
\sigma_{2} \cdots\sigma_{k}$ in a digraph $G$ is defined as a path of the form
$\sigma_{1} \sigma_{2} \dots, \sigma_{j}$, where $0\leq j \leq k$. (For the
case $j=0$, the subpath consists of a single vertex.) A subpath is called
\textit{proper} if $j<k$. A set $S$ of finite directed paths in a directed
graph is called a \textbf{pre-tree} if

\begin{enumerate}
\item every path in $S$ begins at the same vertex $r$, called the
\textit{root} of $S$;
\item no proper subpath of a path in $S$ lies in $S$;
\item for every proper subpath $\sigma$ of a path in $S$ and every edge $e \in
E$, if $e^{-} = \sigma^+$, then the path $\sigma e$ is a subpath
of a path in $S$.
\end{enumerate}
\end{definition}

The subgraph $\langle S \rangle$ of $G$ consisting of all vertices and edges
of $S$ in a pre-tree $S$ may not be a tree. This subgraph $\langle S
\rangle$ may contain cycles, even directed cycles. However, there is an
actually tree $H(S)$ related to $S$ as described in the following proposition.
An example illustrating Proposition~\ref{prop:h} and its proof appears in Figure~\ref{fig:h}.

Graph homomorphism - a map from one graph to another that sends adjacent pairs of vertices to adjacent pairs
of vertices - is a standard notion in graph theory.  For digraphs with possibly multiple edges and loops, we define a
  {\it  homomorphism} from a digraph $G$ to a digraph $G'$ as a pair of functions function $h : V(G) \rightarrow V(G')$ 
and $h : E(G) \rightarrow E(G')$ such that, if edge $e$ is directed from $i$ to $j$, then $h(e)$ 
is directed from $h(i)$ to $h(j)$. We also use the notation $h : G \rightarrow G'$ 
for a digraph homomorphism.  

\begin{proposition}
\label{prop:h} Given a pre-tree $S$, there is a unique pair consisting of a rooted directed tree
$H(S)$ and a homomorphism $h: H(S) \rightarrow\langle S \rangle$ such that 
\begin{enumerate}
\item $h$ takes the root of $H(S)$ to the root of $\langle S \rangle$;
\item $h$ bijectively maps the set of all paths in $H(S)$ going from the root to a leaf onto $S$;
\item among all digraphs satisfying properties (1,2), $H(S)$ has the least
number of edges.
\end{enumerate}
\end{proposition}

\begin{figure}[htb]
\includegraphics[width=9cm, keepaspectratio]{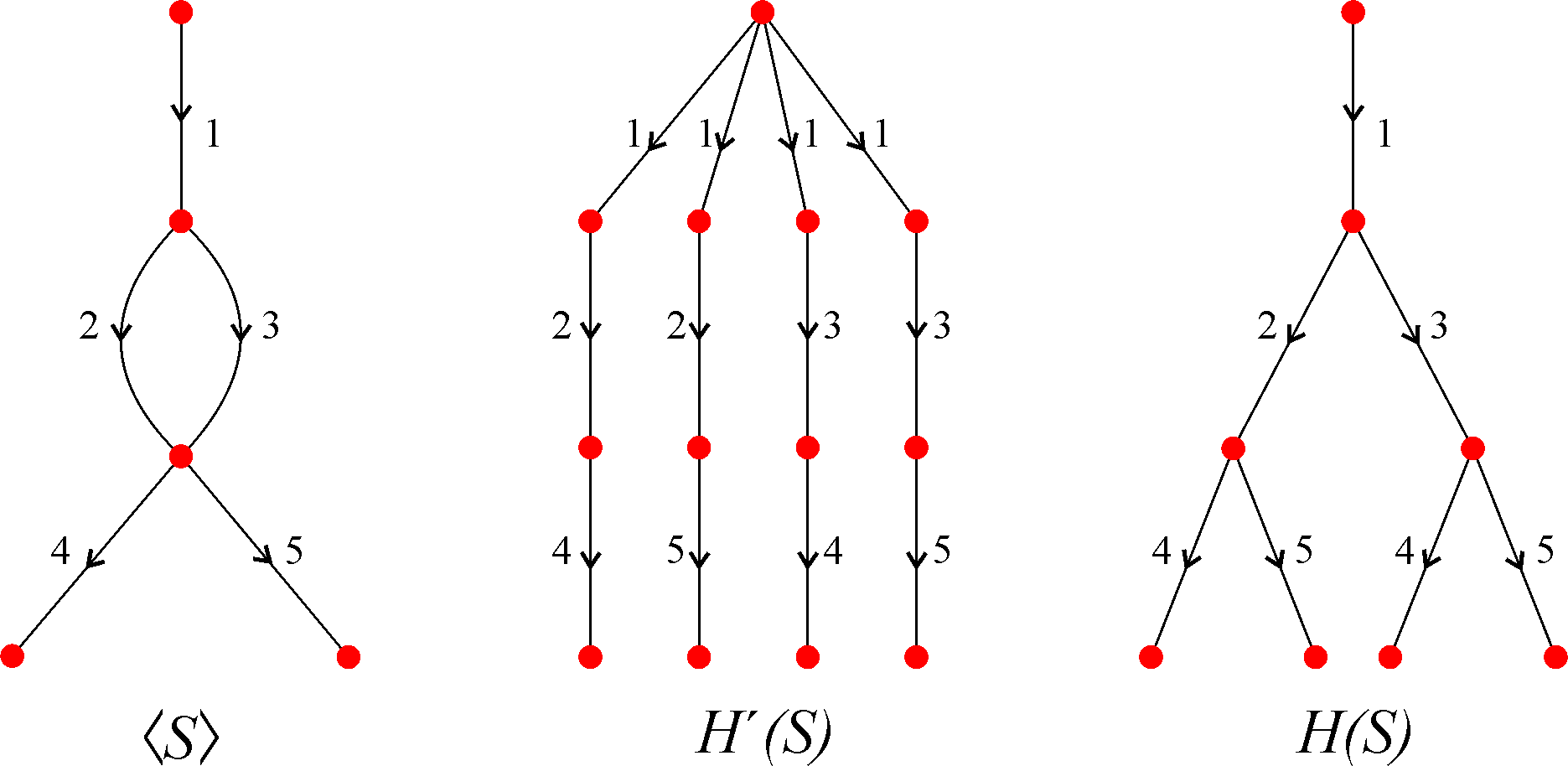}
\caption{A pre-tree  $\langle S \rangle$ as the homomorphic image of a rooted tree $H(S)$.  The intermediate 
tree $H'(S)$ is used in the proof of Proposition~\ref{prop:h}.  
 The set of paths of  the pre-tree $S = \{124, 125, 134, 135\}$.  Edges with
the same label in $H(S)$ are mapped to the same edge of $\langle S \rangle$ by the homomorphism $h$. }
\label{fig:h}
\end{figure}

\begin{proof}  Let $H^{\prime}$ be the rooted tree constructed as follows.  The tree $H^{\prime}$ is the union of directed paths $p_{1}, p_{2}, \dots, p_{m}$
that have  one vertex in common, the root.  Pairwise these paths in $H'$  have no edges in common and no vertices in common except the root.  The edges are
directed from the root to the leaf.  Moreover,  these paths are taken to be exactly all the paths in $S$.  See the two left digraphs in Figure~\ref{fig:h}.   Clearly there is a homomorphism $h^{\prime}: H^{\prime }\rightarrow\langle S \rangle$ satisfying properties (1,2).

The rooted directed tree $H(S)$ is obtained from $H^{\prime}$ as a quotient graph.  Call two vertices $i,j \in V(H')$ equivalent, denoted $i\sim j$, if there are paths $p_{i}$ and $p_{j}$ in $H^{\prime}$, from the root to $i$ and $j$, respectively, such that
 $p_{i}$ and $p_{j}$ have the same image under $h^{\prime}$.  Call two edges $(i_1,j_1)$ and $(i_2,j_2)$ in $E(H')$ equivalent if $i_1 \sim i_2$ and $j_1 \sim j_2$.  Let $H(S)$ be the tree whose vertices and edges are the
equivalence classes of vertices of $H'$.  See Figure~\ref{fig:h}.  Define $h : H(S) \rightarrow \langle S \rangle$ 
by mapping a vertex $v$ to $h'(i)$, where $i\in v$; similarly for the image of edges.  It is routine to check properties (1-3) and uniqueness.
\end{proof}

\subsection{$\theta$-Sequences}

\label{sec:ts}

\begin{definition}
\label{def:ts} Let $\overleftarrow{\theta} := \overleftarrow{\theta}_{1} \;
\overleftarrow{\theta}_{2} \; \overleftarrow{\theta}_{3} \cdots\in
\overleftarrow{\Sigma}^{\infty}$ be an infinite path in $\overleftarrow{G}$,
where $\theta_{i} \in E$ for all $i$, and let $v_{0}, v_{1}, v_{2}, \dots$ be
the successive vertices of ${\overleftarrow{\theta}}$. An infinite sequence
${\mathcal{S}} = \{S_{0}, S_{1}, S_{2},\dots\}$ of pre-trees in $G$ is
called a $\mathbf{\overleftarrow{\theta}}$-\textbf{sequence} if, for all $k =
0,1,2, \dots,$

\begin{enumerate}
\item $v_{k}$ is the root of the pre-tree $S_{k}$, and

\item $\theta_{k} \, S_{k-1} := \{ \theta_{k} \sigma: \sigma\in S_{k-1}\}
\subset S_{k}$ for all $k\geq1$.
\end{enumerate}
The second condition states that the pre-trees along the path $\ot$ are, in a sense, nested.  
Note that edges $\theta_k$ and $\ot _k$ are the same edge oppositely oriented.  
\end{definition}

\subsection{Examples}

In the following examples, the directed graph $G$ is given, as well as a path
$\overleftarrow{\theta}:=\overleftarrow{\theta}_{1}\;\overleftarrow{\theta
}_{2}\;\overleftarrow{\theta}_{3}\cdots\in\overleftarrow{\Sigma}^{\infty}$
with successive vertices $v_{0},v_{1},v_{2},\dots$. Clearly, numerous examples
of ${\overleftarrow{\theta}}$-sequences can be obtained inductively by
choosing an arbitrary pre-tree rooted at $v_{0}$, then a pre-tree rooted
at $v_{1}$ subject to condition (2) in the definition of
${\overleftarrow{\theta}}$-sequence, and continuing at $v_{2},v_{3},\dots$.
Below, we will give a few specific examples that we return to in Section~\ref{sec:examples}.

\begin{example}
\label{ex:1} Let $S_{k}, \, k \geq0,$ be the pre-tree with root $v_{k}$
consisting of all paths of length exactly $k$. Then ${\mathcal{S}} (\theta) =
\{ S_{0}, S_{1}, S_{2}, \dots\}$ is a $\overleftarrow{\theta}$-sequence of pre-trees.
\end{example}

\begin{example}
\label{ex:2} For any $m \in\mathbb{N}$, let $S_{k}, \, k \geq0,$ be the
pre-tree with root $v_{k}$, where
\[
S_{k} = \{ \theta_{k} \theta_{k-1} \cdots\theta_{j+1} \sigma\, : \, 0\leq j
\leq k, \, \sigma^{-} = v_{j}, \;|\sigma| = m\}.
\]
Then ${\mathcal{S}}(\theta, m) := \{ S_{0}, S_{1}, S_{2}, \dots\}$ is a
$\overleftarrow{\theta}$-sequence of pre-trees.
\end{example}

In the next examples, assume that $(G,F)$ is a GIFS where each $f\in F$ is a
similarity transformation. Such a GIFS will be referred to as a
\textbf{similarity GIFS}. For any similarity transformation $f$, let
$\lambda_{f}$ denote the scaling ratio, i.e., $|f(x)-f(y)| = \lambda_{f} \,
|x-y|$ for all $x,y \in{\mathbb{R}}^{d}$.

\begin{example}
\label{ex:3} Given a similarity GIFS, for any pair of positive real numbers $q < Q$, define a
$\overleftarrow{\theta}$-sequence $\mathcal{S }(\theta,q,Q)= \{S_{0}, S_{1},
S_{2}, \dots\}$ by
\[
S_{k} = \{ \sigma\in\Sigma^{*}\, : \, q \leq\lambda( f_{\overleftarrow{\theta
}|k}) \cdot\lambda( f_{\sigma}) \leq Q\}.
\]
Note that, with notation as in equations \eqref{eq:n1} and \eqref{eq:n2} in
Section \ref{sec:notation}, the function $f_{\sigma}$ is a contraction, while
$f_{\overleftarrow{\theta}|k}$ is an expansion. It is routine to check that
condition (3) in the definition of pre-tree and condition (2) in the
definition of ${\overleftarrow{\theta}}$-sequence are satisfied. Therefore
${\mathcal{S}}(\theta, q,Q)$ is a $\overleftarrow{\theta}$-sequence.
\end{example}

\begin{example}
[Balanced Sequences]\label{ex:4} For a similarity GIFS, assume that there
exists a positive real number $s$ and positive integers $d(e)$ such that
\begin{equation}
\label{eq:s}\lambda(f_{e}) = s^{d (e)}%
\end{equation}
for all $e \in E$. Define $d(\overleftarrow{e}) := d(e)$ for all $e\in E$. For
a path $\sigma\in\Sigma^{*}$ or $\sigma\in\overleftarrow{\Sigma}^{*}$, define
\[
d (\sigma) = \sum_{e\in\sigma} d (e ).
\]
For $\sigma= \sigma_{1} \cdots\sigma_{j} \in\Sigma^{*}$, we introduce the
notation $\sigma_{*} := \sigma_{j}$ for the last edge in path $\sigma$. For
$\overleftarrow{\theta} \in{\overleftarrow{\Sigma}}^{\infty}$, define
$\mathcal{S }= \mathcal{S}(\overleftarrow{\theta}) = (S_{0}, S_{1}, S_{2},
\dots)$ by
\begin{equation}
\label{eq:4}S_{k} = \big \{ \sigma\in\Sigma^{*} \, : \, \sigma^{-} = v_{k}
\quad\text{and} \quad0 < d (\sigma) - d (\overleftarrow{\theta}|k) \leq d
(\sigma_{*}) \big \}.
\end{equation}
Intuitively, this guarantees that $d(\sigma)$ does not differ by much from $d
(\overleftarrow{\theta}|k)$.
\end{example}

\begin{definition}
\label{def:balanced} For a given GIFS and $\overleftarrow{\theta}
\in{\overleftarrow{\Sigma}}^{\infty}$, the sequence $\mathcal{S}%
(\overleftarrow{\theta}) = (S_{0}, S_{1}, S_{2}, \dots)$ given by
Equation~\eqref{eq:4} will be called the \textbf{balanced} sequence. The
constant $s$ in Equation~\eqref{eq:s} will be called the \textbf{scaling
constant}.
\end{definition}

\begin{proposition}
\label{prop:balanced} The balanced sequence $\mathcal{S}%
({\overleftarrow{\theta}})$ is a ${\overleftarrow{\theta}}$-sequence.
\end{proposition}

\begin{proof}
We first show that $S_{k}$ is a pre-tree for all $k$. Concerning condition
(2) in the definition of pre-tree, if $\sigma^{\prime}\in S_{k}$ is a
proper subpath of $\sigma\in S_{k}$, then
\[
d(\sigma^{\prime}) - d({\overleftarrow{\theta}}|k) \leq d(\sigma)-d(\sigma
_{*}) - d({\overleftarrow{\theta}}|k) < 0,
\]
which is a contradiction.

Concerning condition (3), if $\sigma^{\prime}$ is a proper subpath of
$\sigma\in S_{k}$ and $e$ is an edge of $G$ such that $e^{-}$ is a vertex on
$\sigma^{\prime}$, then
\[
d(\sigma^{\prime}e) - d({\overleftarrow{\theta}}|k) \leq d(\sigma
)-d(\sigma_{*}) + d(e) - d({\overleftarrow{\theta}}|k) \leq d(e).
\]
Let $\omega$ be a longest path in $G$ starting at $e^{+}$ and satisfying
$d(\sigma^{\prime}e \omega) - d({\overleftarrow{\theta}}|k) \leq d(\omega
_{*})$. It now suffices to show that $d(\sigma^{\prime}e \omega) -
d({\overleftarrow{\theta}}|k) >0$. If $d(\sigma^{\prime}e \omega) -
d({\overleftarrow{\theta}}|k) \leq0$ and $e^{\prime}$ is any edge such that
$e^{-}=\omega^{+}$, then $d(\sigma^{\prime}e \omega e^{\prime}) -
d({\overleftarrow{\theta}}|k) \leq d(e^{\prime}) = (d(\sigma^{\prime}e
\omega))_{*}$, contradicting the maximality of $\omega$.

To verify that $\mathcal{S}({\overleftarrow{\theta}})$ is a
${\overleftarrow{\theta}}$-sequence, assume that $\sigma\in S_{k-1}$. Then $0
< d(\sigma) -d({\overleftarrow{\theta}}|k-1) \leq d(\sigma_{*})$. But
\[
d(\theta_{k}\sigma) - d({\overleftarrow{\theta}}|k) = d(\sigma)
-d({\overleftarrow{\theta}}|k-1) \qquad\text{and} \qquad(\theta_{k}\sigma)_{*}
= \sigma^{*}.
\]
Therefore $\theta_{k}\sigma\in S_{k}.$
\end{proof}

\begin{example}
\label{ex:Am} As a simple special case of Example~\ref{ex:4} , consider the
GIFS depicted in the lower left of Figure~\ref{fig:aman}. It consists of a
single vertex and two loop edges $e_{1}$ and $e_{2}$. Assume that the
corresponding functions have scaling ratios
\[
\lambda(f_{1}) = s\qquad\lambda(f_{2}) = s^{2}
\]
for some $0<s<1$, so that $d(e_{1}) = 1, \, d(e_{2}) = 2$. Let the parameter
be ${\overleftarrow{\theta}} = \overleftarrow{e_{1}} \, \overleftarrow{e_{2}
}\, \cdots$. Then ${\overleftarrow{\theta}}|2 = \overleftarrow{e_{1}}\,
\overleftarrow{ e_{2}}$ and $d(\overleftarrow{\theta}|2) = 3$. The condition
on $\sigma$ in the definition of $S_{2}$ in Equation~\ref{eq:4} is $3 <
d(\sigma) \leq3+d(\sigma_{*})$. The tree $H(S_{2})$ corresponding to the
pre-tree $S_{2}$ at the third vertex $e_{2}^{-}$ of $\theta$, as defined by
Proposition~\ref{prop:h}, is shown at the right in Figure~\ref{fig:aman}. The
labels on each edge $e$ is the values $d(e)$; the label on each leaf $u$ is
the value $d(\sigma)$ of the path $\sigma\in S_{2}$ corresponding to the path
in $H(S_{2})$ from the root to $u$. We will return to this example in Example~\ref{ex:balancedT} in
Section~\ref{sec:examples}, where the red and black colors on the leaves will
be explained.

\begin{figure}[htb]
\includegraphics[width=3cm, keepaspectratio]{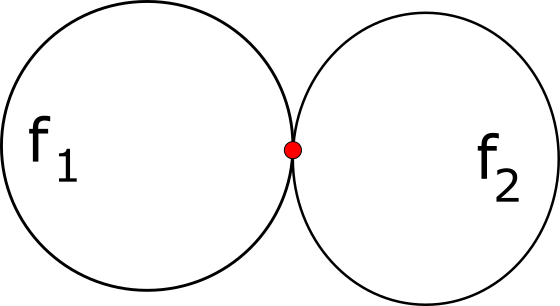} \includegraphics[width=7cm, keepaspectratio]{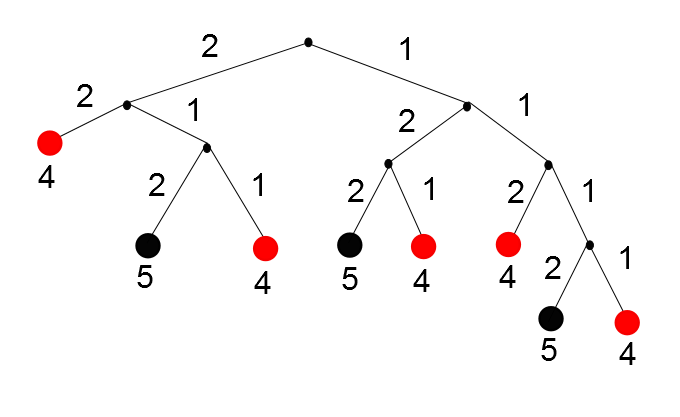}
\caption{The GIFS of Example~\ref{ex:Am} on the left and the tree $H(S_2)$ on the right.}
\label{fig:aman}
\end{figure}

\end{example}

\section{GIFS Tilings}

For this paper, a \textit{tile} is a compact subset of ${\mathbb{R}}^{d}$ with
nonempty interior, and a \textit{tiling of a set} $X \subseteq{\mathbb{R}}%
^{d}$ is a set of pairwise non-overlapping tiles whose union is $X$.

\subsection{The Parameter Space}

\label{sec:ps} Let $(G,F)$ be a GIFS. Any path $\overleftarrow{\theta}%
\in\overleftarrow{\Sigma}^{\infty}$ will be referred to as a
\textbf{parameter} of $G$. Define a metric $d$ on the set
$\overleftarrow{\Sigma}^{\infty}$ of parameters by
\[
d(\sigma,\omega)=%
\begin{cases}
0\quad\text{if}\;\sigma=\omega\\
2^{-k}\quad\text{otherwise, where}\;k\;\text{is the first integer such
that}\;\sigma_{k}\neq\omega_{k}%
\end{cases}
.
\]
This makes $(\Sigma,d)$ a compact metric space, which we call the
\textbf{parameter space} of the GIFS.

\subsection{$\theta$-tilings}

\label{sec:tilings} There are infinitely many parameters for a given GIFS,
and, for every parameter of the GIFS and every corresponding
$\overleftarrow{\theta}$-sequence ${\mathcal{S}} = \{ S_{0}, S_{1}, S_{2},
\dots\}$, a tiling $T(\overleftarrow{\theta}, \mathcal{S})$ will be
constructed as follows. For $\sigma\in S_{k}$:%

\begin{equation}
\label{eq:tiling}\begin{aligned} t(\overleftarrow{\theta}, \mathcal S, k, \sigma) &= f_{\overleftarrow{\theta}|k}\circ f_{\sigma}(A_{\sigma^+}) \\ T(\overleftarrow{\theta}, \mathcal S, k) &= \{ t(\overleftarrow{\theta}, \mathcal S, k, \sigma) \, : \, \sigma \in S_k \} \\ T(\overleftarrow{\theta}, \mathcal S) &= \bigcup_{k=0}^{\infty} T(\overleftarrow{\theta}, \mathcal S, k). \end{aligned}
\end{equation}
Note that, in the first line of Equation~\eqref{eq:tiling}, the map
$f_{\sigma}$ is a contraction, being the composition of contractions, and
$f_{\overleftarrow{\theta}|k}$ is an expansion, being the composition of
inverses of contractions. For each $\sigma\in S_{k}$, the set
$t(\overleftarrow{\theta}, \mathcal{S}, k, \sigma)$ is a single tile. For each
$k$, the set $T(\overleftarrow{\theta}, \mathcal{S}, k)$ is a patch of tiles.
These patches are nested, i.e., $T(\overleftarrow{\theta}, \mathcal{S}, k)
\subset T(\overleftarrow{\theta}, \mathcal{S}, k+1)$ for all $k$ because, for
$\sigma\in S_{k}$, we have
\[
f_{\overleftarrow{\theta}|k}\circ f_{\sigma}(A_{\sigma^{+}}) =
f_{\overleftarrow{\theta}|k+1} \circ f_{\theta_{k+1}} f_{\sigma}(A_{\sigma
^{+}}) = f_{\overleftarrow{\theta}|k+1} \circ f_{\omega}(A_{\omega^{+}}),
\]
where $\omega= f_{\theta_{k+1}} f_{\sigma} \in S_{k+1}$ by condition (2) in
the definition of $\overleftarrow{\theta}$-sequence. The tiling
$T(\overleftarrow{\theta}, \mathcal{S})$ is the nested union of the patches
$T(\overleftarrow{\theta}, \mathcal{S}, k)$.

The tiles in a patch $T(\overleftarrow{\theta}, \mathcal{S}, k)$ are in
bijection with the leaves of the tree $H(S_{k})$ as defined in
Proposition~\ref{prop:h}. Specifically, $T(\overleftarrow{\theta},
\mathcal{S}, k) = \{ f_{\overleftarrow{\theta}|k}\circ f_{\sigma}%
(A_{\sigma^{+}}) \, : \, \sigma\in P \}$, where $P$ is the set of paths in
$H(S_{k})$ from the root to a leaf.

\begin{definition}
Given a tiling parameter $\overleftarrow{\theta} \in\overleftarrow{\Sigma
}^{\infty}$ of a GIFS and a $\overleftarrow{\theta}$-sequence $\mathcal{S}$,
the tiling $T(\overleftarrow{\theta}, \mathcal{S})$ will be called the
$(\overleftarrow{\theta},\mathcal{S})$-\textbf{tiling} or, if $\mathcal{S}$ is
understood, then simply the $\overleftarrow{\theta}$-\textbf{tiling}.
\end{definition}

For the set $T(\overleftarrow{\theta},\mathcal{S})$ of tiles, denote by
$\langle T(\overleftarrow{\theta},\mathcal{S})\rangle$ the union of these
tiles. Call $\overleftarrow{\theta}$ \textbf{filling} if $\langle
T(\overleftarrow{\theta},\mathcal{S})\rangle={\mathbb{R}}^{d}$ for all
$\overleftarrow{\theta}$-sequences $\mathcal{S}$. Theorem~\ref{thm:filling}
states that almost all parameters are filling. A proof is omitted since it is
similar to that of \cite[Theorems 3.7 and 3.9]{BV1}. Theorem~\ref{thm:tiling}
states that if ${\overleftarrow{\theta}}$ is filling, then
$T(\overleftarrow{\theta},\mathcal{S})$ is indeed a tiling of ${\mathbb{R}%
}^{d}$.

\begin{theorem}
\label{thm:filling} For a GIFS $(G,F)$, almost all $\overleftarrow{\theta}
\in\overleftarrow{\Sigma}^{\infty}$ are filling in the following sense of
``almost all":

\begin{enumerate}
\item The set of filling $\overleftarrow{\theta} \in\overleftarrow{\Sigma
}^{\infty}$ is dense in the parameter space $\overleftarrow{\Sigma}^{\infty}$.

\item Suppose that a positive probability is assigned to each edge such that the sum of the
probabilities of the out-edges at each vertex is one, and
suppose that a parameter in $\overleftarrow{\Sigma}^{\infty}$ is chosen at
random according to these probabilities. Then, with probability 1, this random
parameter is filling.
\end{enumerate}
\end{theorem}

\begin{lemma}
\label{lem:overlap} For any GIFS and any pre-tree $S$, the set
\[
W(S) := \{ f_{\sigma}(A_{\sigma^{+}}) : \sigma\in S\}
\]
is a tiling of the attractor component $A_{r}$, where $r$ is the root of the
pre-tree $S$.
\end{lemma}

\begin{proof}
The proof is by induction on the number of paths $|S|$ in the pre-tree $S$.
For $|S|=1$, the pre-tree consists of a single vertex, and the set $W$
consists of the single tile $A_{r}$. For $|S| \geq1$, assume the assertion is
true for all pre-trees with fewer paths. Prune $S$ as follows to obtain a
smaller pre-tree $S^{\prime}$. Let $\sigma= \sigma_{1} \, \sigma_{2}
\cdots\sigma_{m}$ be a longest path in $S$; let $\sigma^{\prime}:= \sigma_{1}
\, \sigma_{2} \cdots\sigma_{m-1}$; and let $x$ be the last vertex on
$\sigma^{\prime}$. Then $S^{\prime\prime}:= \{ \sigma^{\prime}e \,: \, e \in
E_{x}\}$ is a subset of $S$. Since the set of paths $S^{\prime}:= S \setminus
S^{\prime\prime}\cup\{\sigma^{\prime}\}$ is a pre-tree, by the induction
hypothesis, $W(S^{\prime})$ is a tiling of $A_{r}$. Now
\[
\begin{aligned}
W(S) &= \{ f_{\sigma}(A_{\sigma^+}) : \sigma \in S \setminus S''\}\cup
\{ f_{\sigma} ( f_e (A_{e^+}) )\, : \, e \in E_x\} \\
W(S') &=  \{ f_{\sigma}(A_{\sigma^+}) : \sigma \in S \setminus S''\} \cup
\{ f_{\sigma'}(A_{x}) \}.
\end{aligned}
\]
But by Equation~\eqref{eq:decomp}, the set $\{ f_{e}(A_{e^{+}}) \, : \, e \in
E_{x}\}$ is a tiling of $A_{x}$. Therefore $W(S)$ is a also tiling of $A_{r}$.
\end{proof}

\begin{theorem}
\label{thm:tiling} Given a GIFS and a parameter $\overleftarrow{\theta}
\in\overleftarrow{\Sigma}^{\infty}$, if $\overleftarrow{\theta}$ is filling,
then $T(\overleftarrow{\theta}, \mathcal{S})$ is a tiling of ${\mathbb{R}}%
^{d}$ for every $\overleftarrow{\theta}$-sequence $\mathcal{S }= \{S_{0},
S_{1}, \dots\}$.
\end{theorem}

\begin{proof}
Since $\overleftarrow{\theta}$ and $\mathcal{S}$ are fixed throughout the
proof, we omit them in the notation: let $T(k) := T(\overleftarrow{\theta},
\mathcal{S}, k), \, T := T(\overleftarrow{\theta}, \mathcal{S})$. Because it
is assumed that $\overleftarrow{\theta}$ is filling, $\langle T \rangle=
{\mathbb{R}}^{d}$. Hence, to show that $T$ is a tiling of ${\mathbb{R}}^{d}$
it suffices to show that pairs of distinct tiles in $T$ do not overlap. This
is equivalent to showing that, for all $k$, all pairs of distinct tiles in
$T(k)$ do not overlap. By definition, the set of tiles
\[
T(k) = f_{\overleftarrow{\theta}|k} \{f_{\sigma} (A_{\sigma^{+}}) : \sigma\in
S_{k}\} = f_{\overleftarrow{\theta}|k} W(S_{k}).
\]
By Lemma~\ref{lem:overlap}, distinct tiles in $W(S_{k})$ do not overlap. Since
the map $f_{\overleftarrow{\theta}|k}$ is a homeomorphism, distinct tiles in
$T(k)$ also do not overlap.
\end{proof}

\subsection{$\theta$-Tiling Examples}  \label{sec:examples} 
The examples in this section correspond to the four
examples of ${\overleftarrow{\theta}}$-sequences in Section~\ref{sec:ts}.

\begin{example}
[Tilings by Squares]\label{ex:squares} Consider the GIFS in ${\mathbb{R}}^{2}$
where the graph $G$ consists of a single vertex and four loops, which will be
referred to as edges $1,2,3,4$, and corresponding functions assigned to these
loops:
\[
\begin{aligned}  f_1(x) &= \frac12 x \\ f_2(x)&= \frac12 x + \Big (\frac12 , 0\Big) \end{aligned} \qquad
\qquad\begin{aligned}  f_3(x)&= \frac12 x + \Big(\frac12,  \frac12\Big )
\\ f_4(x)&= \frac12 x + \Big(0, \frac12 \Big).
\end{aligned}
\]
The attractor has only one component, which is a square.

If, for a given parameter $\theta$, the corresponding ${\overleftarrow{\theta
}}$-sequence $\mathcal{S}(\theta)$ is the one given in Example~\ref{ex:1} of
Section~\ref{sec:ts}, then the $(\theta, \mathcal{S}(\theta))$-tiling is the
standard tiling of the plane by unit squares, independent of $\theta$.

If, on the other hand, the corresponding ${\overleftarrow{\theta}}$-sequence
$\mathcal{S}(\theta,m)$ is the one given in Example~\ref{ex:2} of
Section~\ref{sec:ts}, then taking, for example, $\overleftarrow{\theta} =
\overleftarrow{1} \,\overleftarrow{2} \,\overleftarrow{3} \,\overleftarrow{4}
\,\overleftarrow{1} \,\overleftarrow{2} \,\overleftarrow{3} \,
\overleftarrow{4} \cdots$ and $m=1$, a patch of the spiral-like $(\theta,
\mathcal{S}(\theta,1))$-tiling is shown in Figure~\ref{fig:squares}.
Progressing outward, the squares increase in size.

\begin{figure}[htb]
\centering
\includegraphics[width=6cm, keepaspectratio]{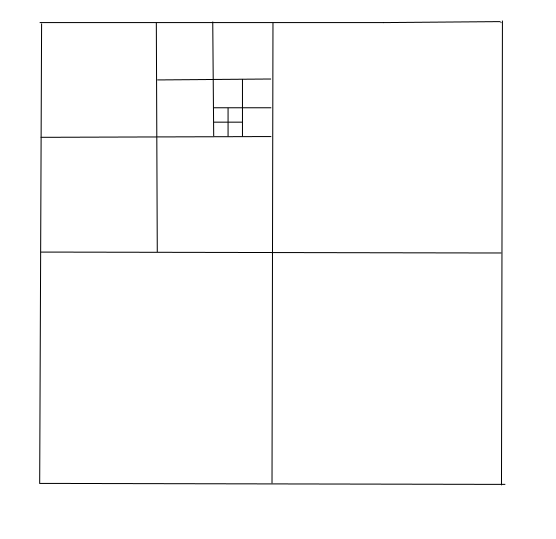}
\vskip -3mm
\caption{A patch of the tiling of Examples~\ref{ex:3} and \ref{ex:squares}.}
\label{fig:squares}
\end{figure}

\end{example}

\begin{example}
[Uniform Tilings]Call a tiling \textit{uniform} if there are positive real
numbers $r,R$ such that each tile contains a ball of radius $r$ and is
contained in a ball of radius $R$. Given a parameter $\overleftarrow{\theta}
\in{\overleftarrow{\Sigma}}^{\infty}$ for a similarity GIFS, and taking the
corresponding ${\overleftarrow{\theta}}$-sequence $\mathcal{S}(\theta, q, Q)$
of Example~\ref{ex:3} of Section~\ref{sec:ts}, it is routine to verify that
$T(\overleftarrow{\theta},\mathcal{S}(q,Q) )$ is a uniform tiling. The tiling
that appears in Figure~\ref{fig:uniform} is uniform. Although not apparent in
this finite patch, there are infinitely many tile shapes up to congruence, but
finitely many up to similarity.

\begin{figure}[htb]
\centering
\includegraphics[width=6cm, keepaspectratio]{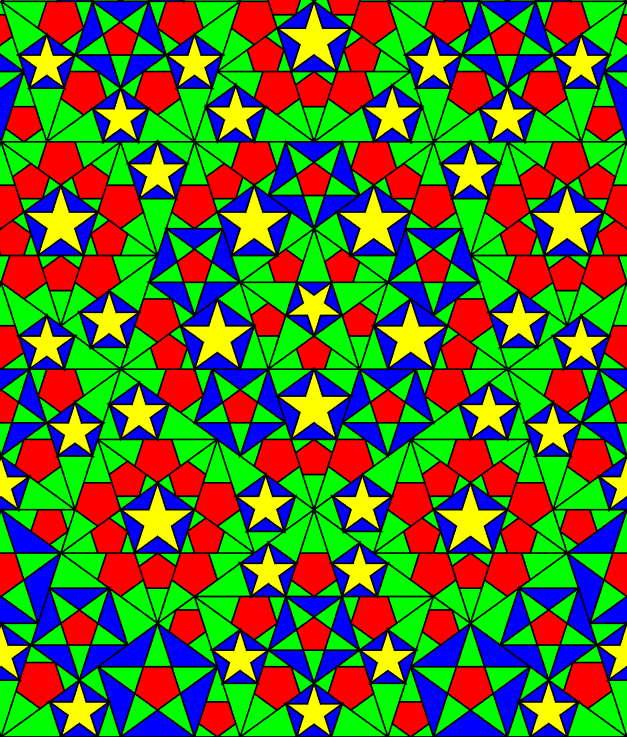}
\vskip -3mm
\caption{A uniform tiling.}
\label{fig:uniform}
\end{figure}

\end{example}

\begin{example}
[Balanced Tilings]\label{ex:balancedT}

\begin{definition}
\label{def:balanced2} For a given similarity GIFS and filling parameter
${\overleftarrow{\theta}}$, the tiling $T(\ot) :=T({\overleftarrow{\theta}},
\mathcal{S}(\theta))$, where $\mathcal{S}({\overleftarrow{\theta}})$ is the
balanced ${\overleftarrow{\theta}}$-sequence (as in
Definition~\ref{def:balanced}), will be referred to as the \textbf{balanced
GIFS tiling}.
\end{definition}

\vskip 2mm

The \textit{order} of a tiling is the number of tiles up to congruence.

\begin{proposition}
A balanced GIFS tiling has finite order, which is at most $\max\{d(e): e \in
E\}$.
\end{proposition}

\begin{proof}
Let $d_{\max} = \max\{d(e): e \in E\}$. Any tile $t \in
T(\overleftarrow{\theta}, \mathcal{S})$ is, by definition, congruent to
$\lambda\,A$, where $A$ is a component of the attractor and $\lambda=
s^{d({\sigma} ) -d ({\overleftarrow{\theta}|k})}$, where $\sigma\in S_{k}$ for
some $k$. By the definition of $S_{k}$ as given in Equation~\eqref{eq:4}, we
have
\[
0 <d({\sigma} ) - d ({\overleftarrow{\theta}|k}) \leq d_{\max}.
\]
Therefore there are at most finitely many possibilities for $\lambda$.
\end{proof}

A balanced tiling corresponding to the GIFS in Figure~\ref{fig:aman} in
Example~\ref{ex:Am} is shown in Figure~\ref{fig:aman2}. In this case the
scaling constant $s$ in Definition~\ref{def:balanced} is the square root of
the golden ratio $(1+\sqrt{5})/2$. In this tiling there are two similar Ammann
tiles, a small one and a large one. In the tree $H(S_{2})$ in
Figure~\ref{fig:aman}, the red leaves correspond to the large tiles in the
patch $T({\overleftarrow{\theta}},\mathcal{S}, 2)$ and the black leaves
correspond to the small tiles.

In \cite{bandtTILE} Bandt and Gummelt construct fractal versions of the Penrose kite and dart.  A balanced
tiling construction of a fractal Penrose tiling is shown in Figure~\ref{fig:fp}.  The directed graph consisting of
two vertices and five edges (3 loops) is shown on the left.   The scaling ratio of each of the
five functions (whose formulas are omitted) is the reciprocal of the golden ratio.  

The tilings in Figures~\ref{fig:fo1}, \ref{fig:fo2}, and \ref{fig:fo3}, as
well as those in Figures~\ref{fig:a} and \ref{fig:b}, are also examples of
balanced GIFS tilings.
\end{example}

\begin{figure}[htb]
\vskip -3mm
\includegraphics[width=8cm, keepaspectratio]{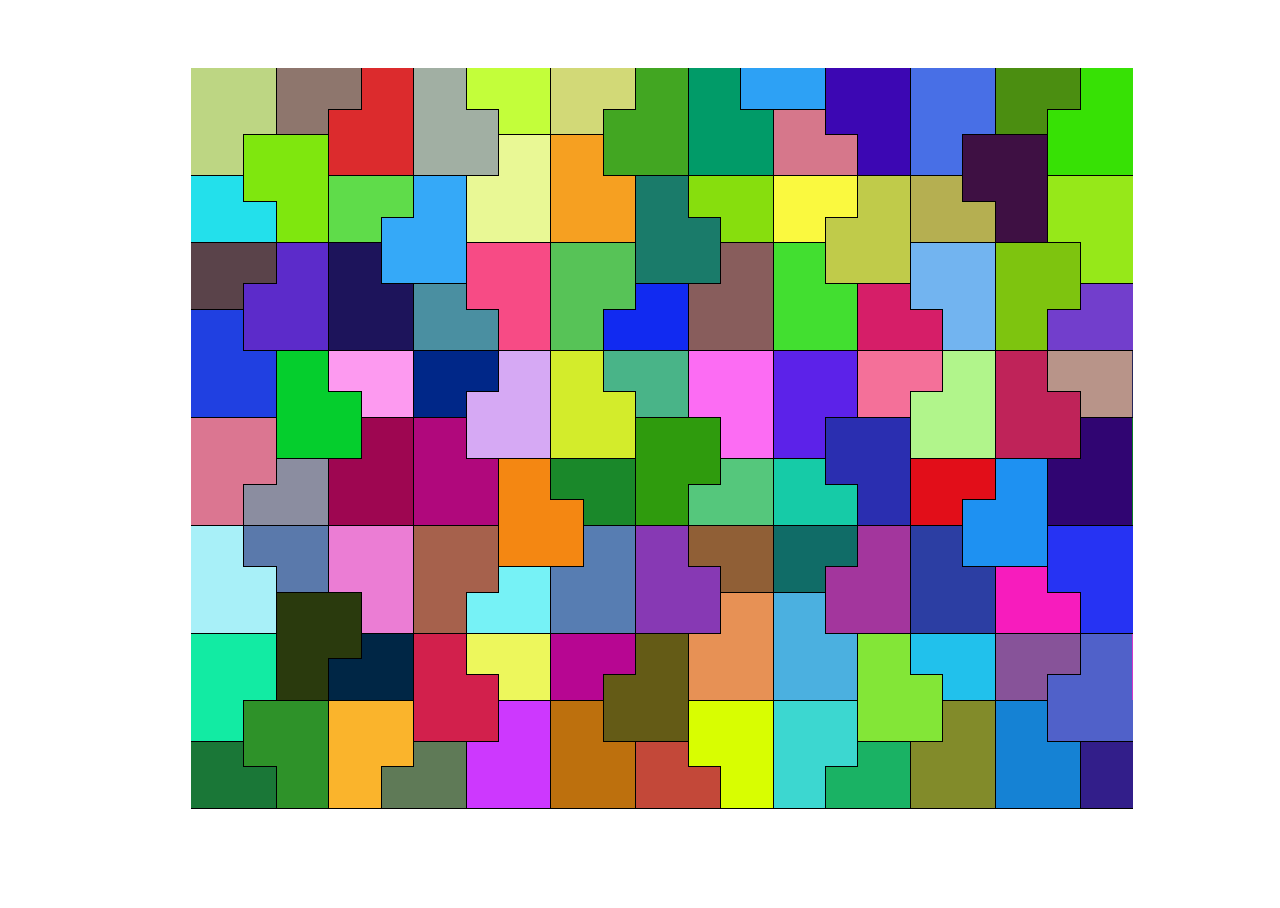}
\vskip -5mm
\caption{A balanced tiling using tiles due to R. Ammann.}
\label{fig:aman2}
\end{figure}

\begin{figure}[htb]
\includegraphics[width=5cm, keepaspectratio]{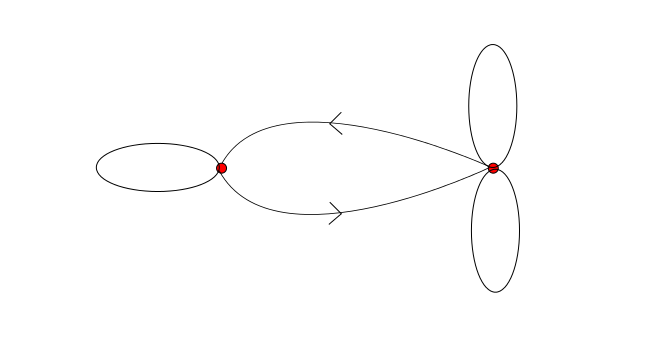}   \includegraphics[width=7cm, keepaspectratio]{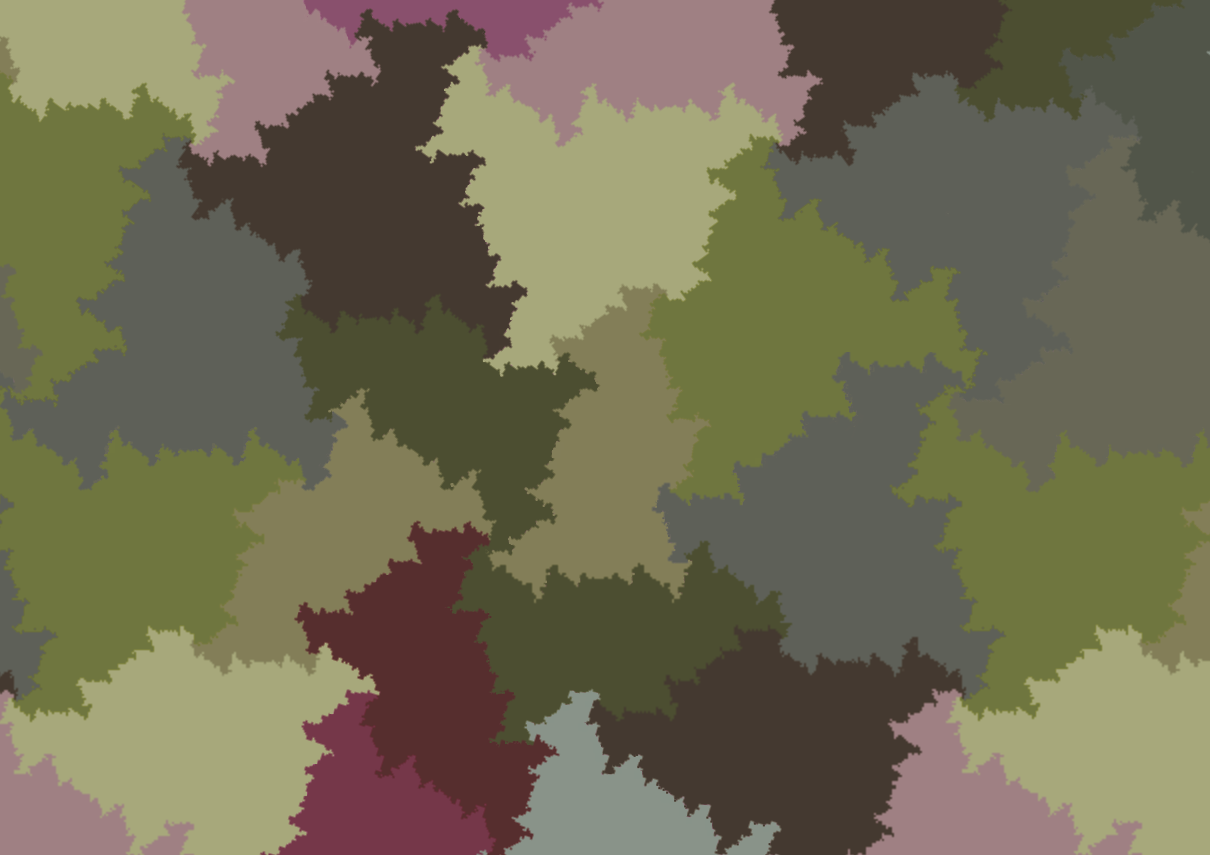} 
\caption{A fractal Penrose tiling.}
\label{fig:fp}
\end{figure}

\begin{figure}[htb]
\vskip -3mm
\includegraphics[width=9.5cm, keepaspectratio]{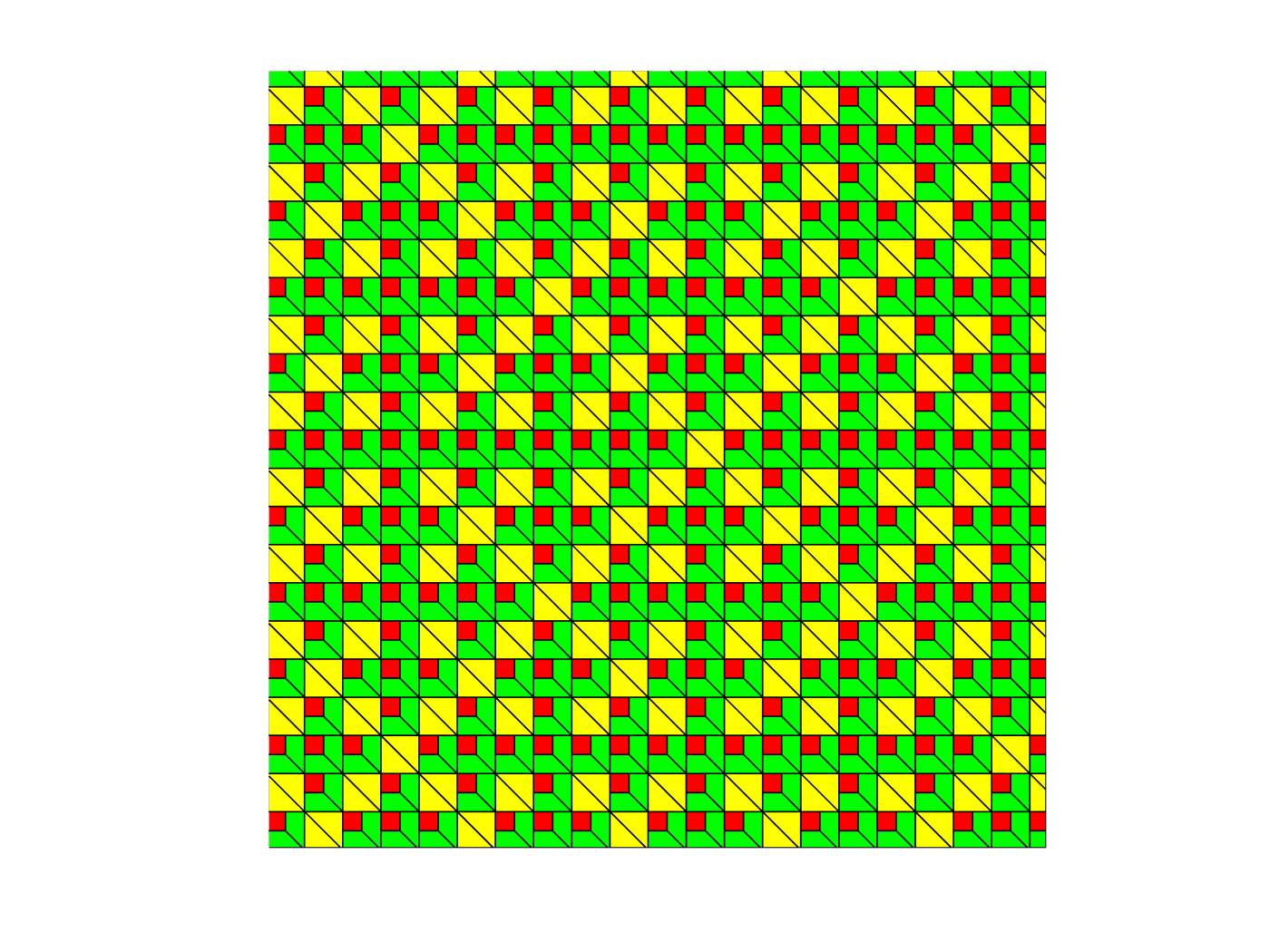}
\vskip -5mm
\caption{A balanced GIFS tiling.}
\label{fig:fo1}
\end{figure}

\begin{figure}[htb]
\begin{center}
\includegraphics[width=5cm, keepaspectratio]{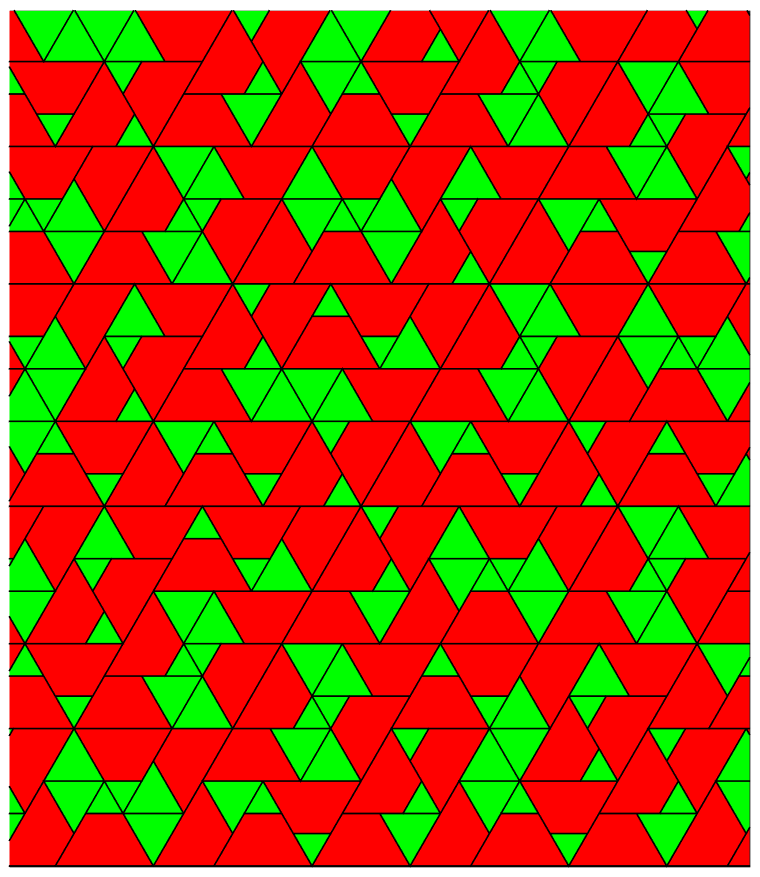}  \hskip 7mm
\includegraphics[width=5cm, keepaspectratio]{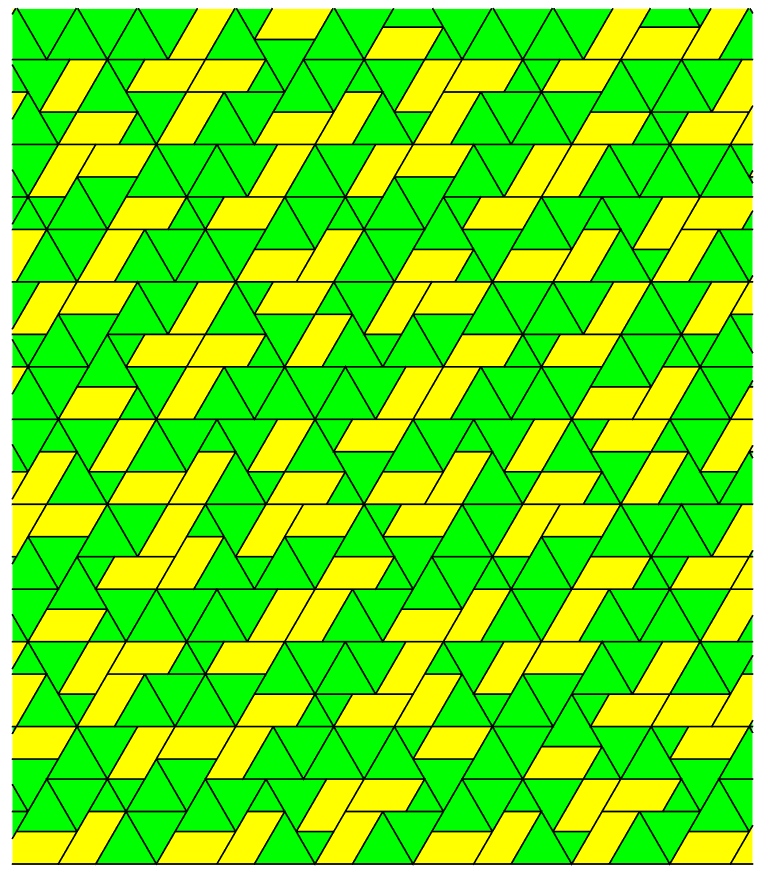}
\end{center}
\caption{Balanced GIFS tilings from the same GIFS.}
\label{fig:fo2}
\end{figure}

\begin{figure}[htb]
\begin{center}
\includegraphics[width=6cm, keepaspectratio]{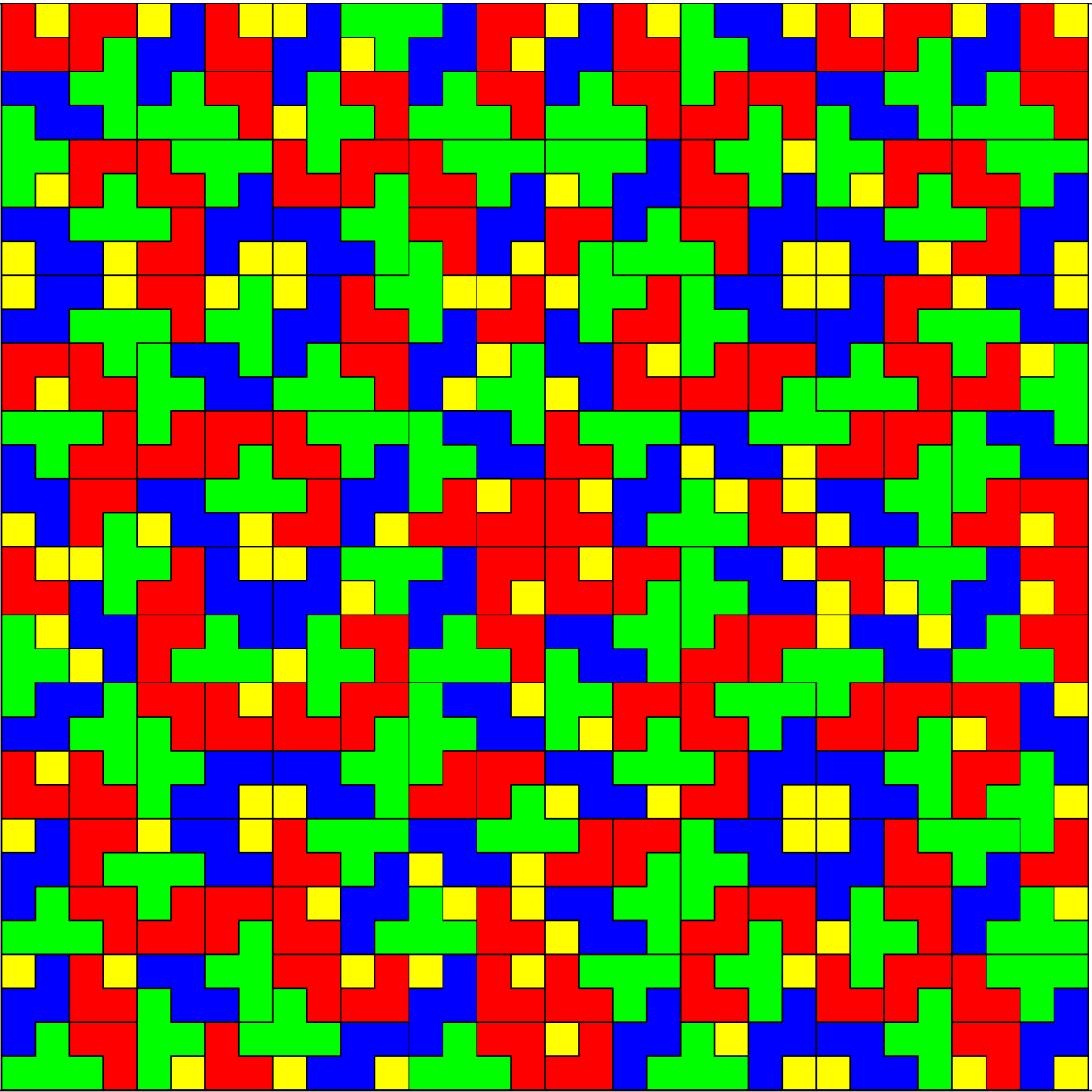} 
\end{center}
\caption{A balanced GIFS tiling.}
\label{fig:fo3}
\end{figure}

\subsection{Tile Addresses}

\label{sec:a} Assume that the parameter ${\overleftarrow{\theta}}$ of a GIFS
is fixed. Each tile in $t \in T(\overleftarrow{\theta}, \mathcal{S})$ can be
uniquely expressed as $t = f_{\overleftarrow{\theta}|k}\circ f_{\sigma
}(A_{\sigma^{+}})$, where $\theta_{k}^{-} \neq\sigma_{1}$. Define the unique
\textbf{address} of tile $t$ to be $k\bullet\sigma$. The proposition below
follows from the definition of the address and Equation~\eqref{eq:3}.

\begin{proposition}
If the address of a tile $t \in T(\overleftarrow{\theta}, \mathcal{S})$ is
$k\bullet\sigma$, then
\[
t = \Big \{ f_{\overleftarrow{\theta}|k} \big (\lim_{j \rightarrow\infty}
f_{\sigma\alpha|j}(x_{0}) \big )  \, : \, \alpha\in\Sigma^{\infty}, \,
\alpha^{-} = \sigma^{+} \Big \},
\]
where the limit is independent of the point $x_{0} \in{\mathbb{R}}^{d}$.
\end{proposition}

\section{Tiling Hierarchy}

\label{sec:hier}

Given a parameter ${\overleftarrow{\theta}}$ of a GIFS and two
${\overleftarrow{\theta}}$-sequences ${\mathcal{S}} = \{S_{0}, S_{1},
S_{2},\dots\}$ and ${\mathcal{S }^{\prime}} = \{S^{\prime}_{0}, S^{\prime}%
_{1}, S^{\prime}_{2},\dots\}$ such that $\mathcal{S }\neq\mathcal{S}^{\prime}%
$, we say that ${\mathcal{S}^{\prime}}$ \textit{lies above} $\mathcal{S}$ if,
for every $k$, each path in $S^{\prime}_{k}$ is a subpath of a path in $S_{k}$.

\begin{lemma}
\label{lem:hier} If ${\mathcal{S}^{\prime}}$ lies above $\mathcal{S}$, then
each tile in $T({\overleftarrow{\theta}}, \mathcal{S}^{\prime})$ is tiled by
tiles in $T({\overleftarrow{\theta}}, \mathcal{S})$.
\end{lemma}

\begin{proof}
It is sufficient to show that each $t \in T({\overleftarrow{\theta}},
\mathcal{S})$ is contained in a tile $t^{\prime}\in T({\overleftarrow{\theta}%
}, \mathcal{S}^{\prime})$. Let $t = f_{\overleftarrow{\theta}|k} \circ
f_{\sigma}(A_{\sigma^{+}}) \in T({\overleftarrow{\theta}}, \mathcal{S})$.
Since ${\mathcal{S}^{\prime}}$ lies above $\mathcal{S}$, there is a path
$\omega\in S^{\prime}_{k}$ such that $\omega$ is a subpath of $\sigma$. We
claim that $t \in f_{\overleftarrow{\theta}|k} \circ f_{\omega} (A_{\omega
^{+}}) \in T({\overleftarrow{\theta}}, \mathcal{S }^{\prime})$, which would
complete the proof. The claim, however, follows in same way as the proof of
Lemma~\ref{lem:overlap}.
\end{proof}

\begin{definition}
Given a sequence $\mathcal{S}_{0}, \mathcal{S}_{1}, \mathcal{S}_{2}, \dots$ of
${\overleftarrow{\theta}}$-sequences such that $\mathcal{S}_{k+1}$ lies above
$\mathcal{S}_{k}$ for all $k$, the sequence of tilings
\[
T({\overleftarrow{\theta}}, \mathcal{S}_{0}),T({\overleftarrow{\theta}},
\mathcal{S}_{1}),T({\overleftarrow{\theta}}, \mathcal{S}_{2}), \dots
\]
is called a \textbf{hierarchy} for the tiling $T({\overleftarrow{\theta}}) :=
T({\overleftarrow{\theta}}, \mathcal{S}_{0})$. The term ``hierarchy" is used
because, as follows from Lemma~\ref{lem:hier}, each tile in
$T({\overleftarrow{\theta}}, \mathcal{S}_{k})$ is contained in a tile in
$T({\overleftarrow{\theta}}, \mathcal{S}_{k+1})$, for all $k$.
\end{definition}

\subsection{Hierarchies for a Balanced GIFS Tiling}

For a fixed similarity GIFS, fix a parameter $\overleftarrow{\theta}
\in{\overleftarrow{\Sigma}}^{\infty}$. Let $\mathcal{S}$ denote the balanced
${\overleftarrow{\theta}}$-sequence as in Definition~\ref{def:balanced}, so
that $T({\overleftarrow{\theta}}, \mathcal{S})$ is the corresponding balanced
GIFS tiling. We now introduce two particular hierarchies in the balanced case.

For integers $n,k\geq0$, let
\[
\begin{aligned}
S_{(n,k)} &=  \big \{ \sigma  \in \Sigma^* \, : \,  \sigma^- = \theta_k^- \quad \text{and} \quad
0 <  d (\sigma)-d (\oo \theta|k) + d (\oo \theta|n) \leq  d (\sigma_*) \big \} \\
\widehat S_{(n,k)} &=  \big \{ \sigma  \in \Sigma^* \, : \,  \sigma^- = \theta_k^- \quad \text{and} \quad
0 <  d (\sigma)-d (\oo \theta|k) +n \leq  d (\sigma_*) \big \}
\end{aligned}
\]
and
\[
\begin{aligned} {\mathcal S}_n &:= \mathcal S_n (\oo\theta) = (S_{n,0}, S_{n,1}, S_{n,2}, \dots ) \\
{\widehat{\mathcal S}}_n &:=	\widehat{\mathcal S}_n (\oo\theta) = (\widehat S_{n,0}, \widehat S_{n,1}, \widehat S_{n,2}, \dots ) .
\end{aligned}
\]
Note that $\mathcal{S}_{0} = \mathcal{S}$ and, for $n\geq1$, the first
few terms in the above sequences may be empty. As in
Proposition~\ref{prop:balanced}, it is not hard to show that ${\mathcal{S}}%
_{n}$ and ${\widehat{\mathcal{S}}}_{n}$ are $\overleftarrow{\theta}%
$-sequences. And, according to Proposition~\ref{prop:hier} below, the
corresponding sequences of tilings
\[
\mathcal{T }:= (T_{0}, T_{1}, T_{2}, \dots) \qquad\text{and} \qquad
\widehat{\mathcal{T}} := (\widehat{T}_{0}, \widehat{T}_{1}, \widehat{T}_{2},
\dots),
\]
where $T_{n} = T(\overleftarrow{\theta}, \mathcal{S}_{n})$ and $\widehat{T}%
_{n} = \widehat{T}(\overleftarrow{\theta}, \widehat{\mathcal{S}}_{n})$, are
tiling hierarchies. Call these \textbf{balanced tiling hierarchies}. The
subscript $n$ is referred to as the \textbf{level} of the tiling in the hierarchy. For the
chair tiling shown in Figure~\ref{fig:chair}, the balanced hierarchies
$\mathcal{T}$ and $\widehat{\mathcal{T}}$ are the same, the first three levels
of the hierarchy shown using thicker lines in Figure~\ref{fig:chair}.

\begin{definition}
\label{def:bs} For a similarity GIFS and for each $i = 1,2, \dots, N$, let $B_{i} := \{f_{e}(A_{e^{+}}) : e \in E_{i}\}$.
These sets are fundamental objects of the GIFS because the $i^{\text{th}}$ component $A_i$ of the attractor
is, by definition, $A_i = \bigcup \{ X : X\in B_i\}$.   Call any set congruent to 
$s^j \, B_i :=   \{  s^j \, f_e (A_{e+}) : e\in E_i \}$, where $s$ is the scaling constant, a \textbf{basic}
$i$-\textbf{subdivision}.  For example, for the tiling in Figure~\ref{fig:a}, the four basic subdivisions, up to a
similarity,
are shown in Figure~\ref{fig:basic}.
\end{definition}

\begin{figure}[htb]
\includegraphics[width=3cm, keepaspectratio]{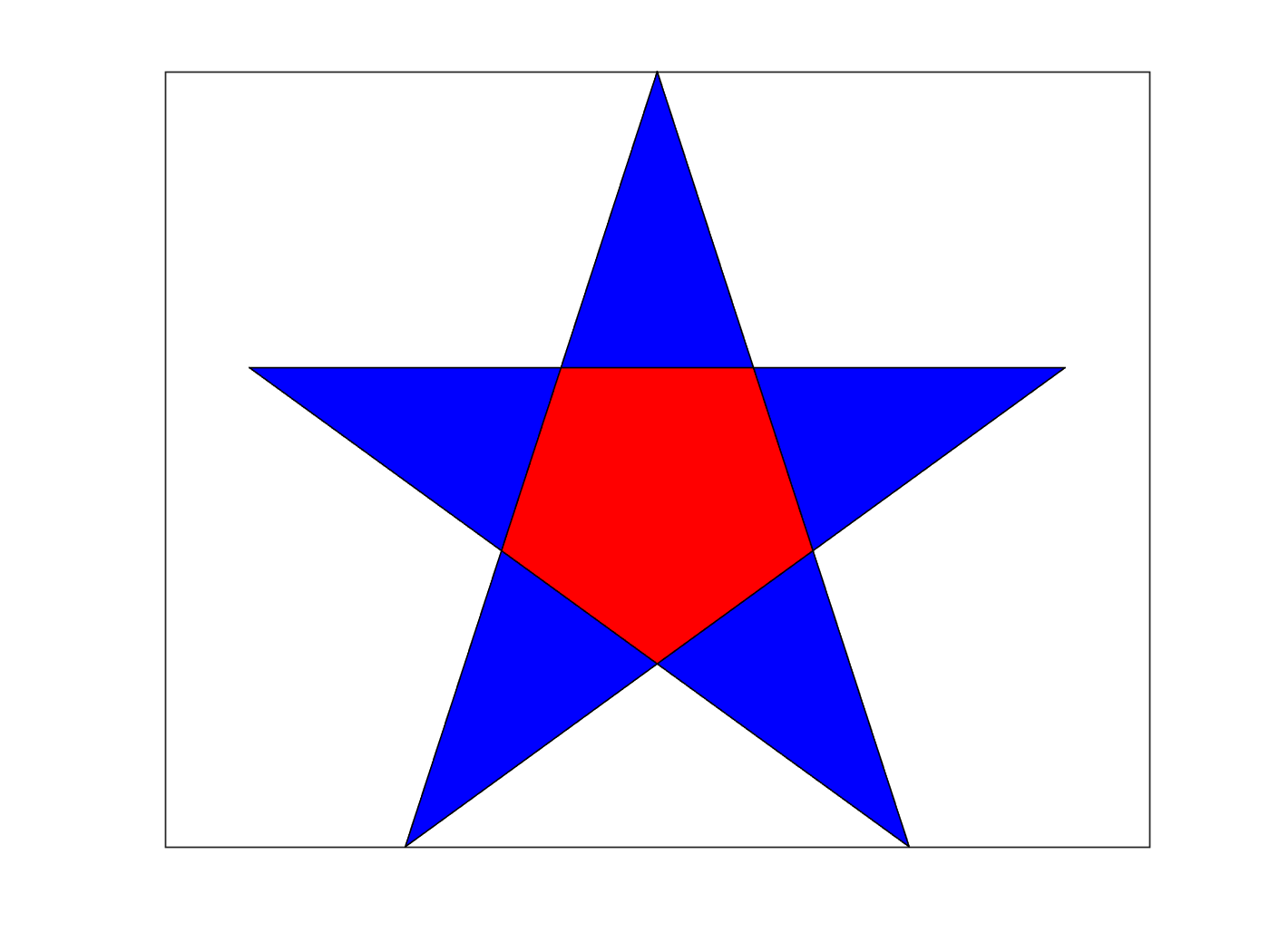} \hskip 1mm\includegraphics[width=3cm, keepaspectratio]{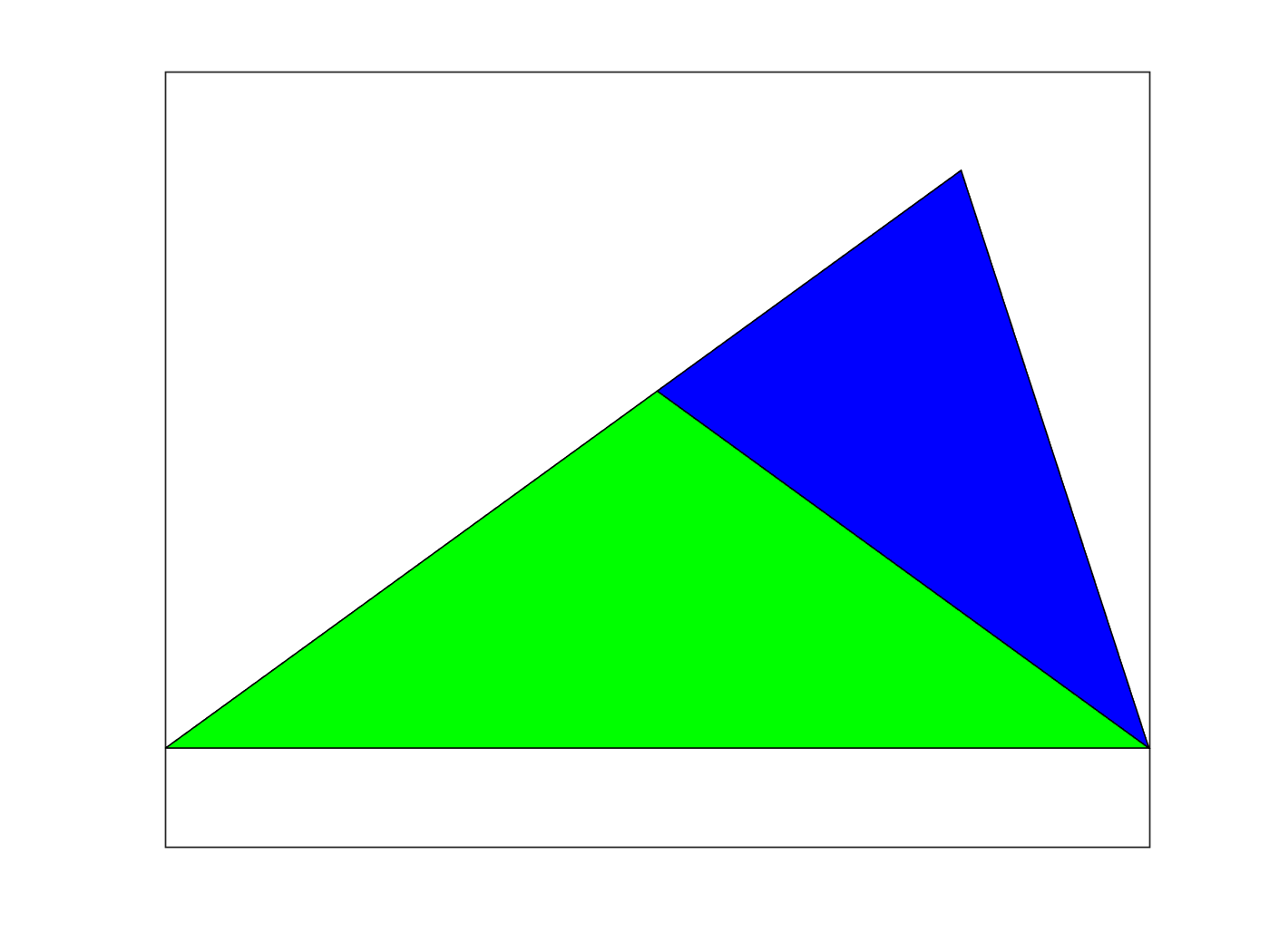} \hskip 1mm \includegraphics[width=3cm, keepaspectratio]{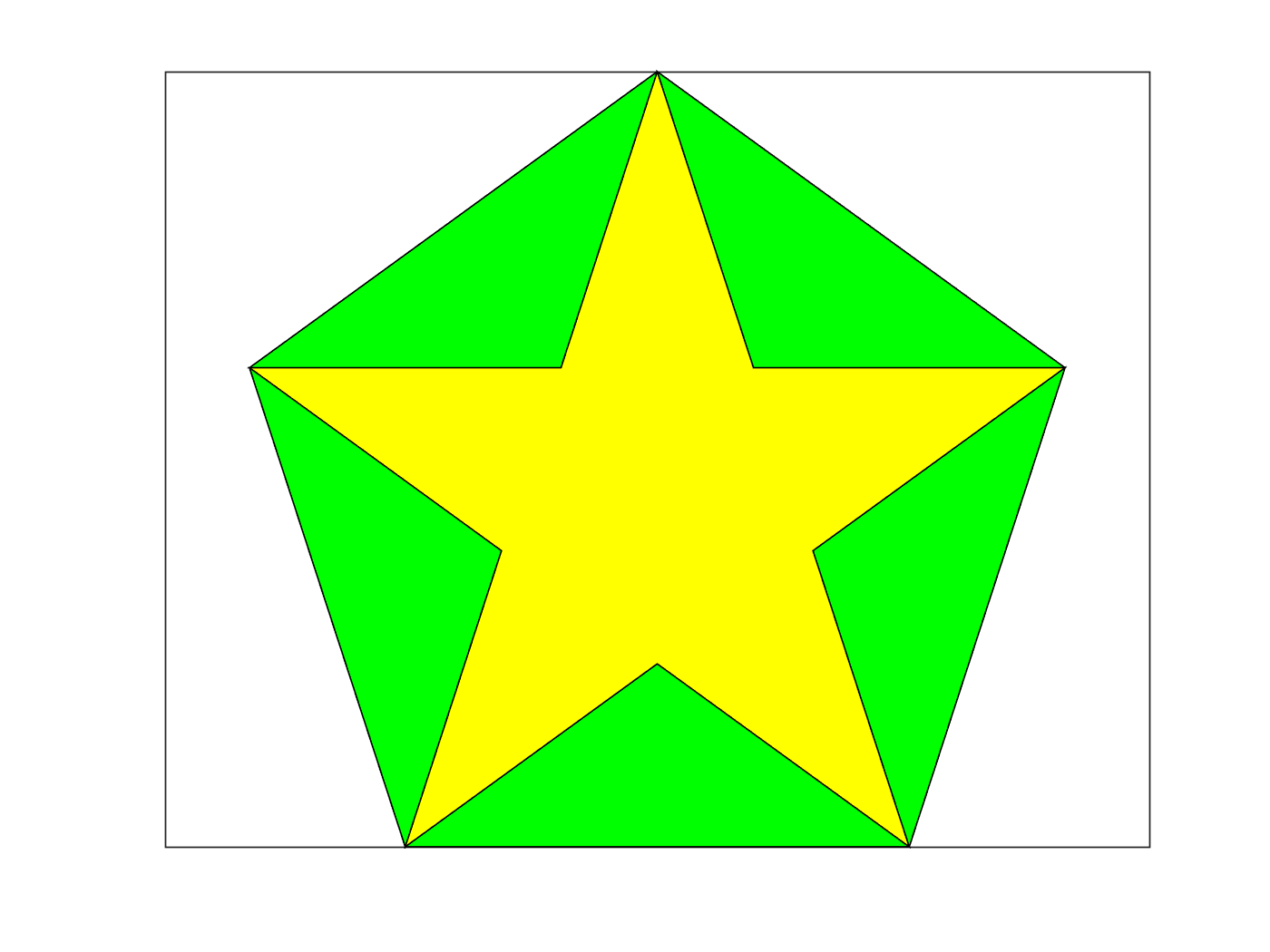} \hskip 1mm\includegraphics[width=3cm, keepaspectratio]{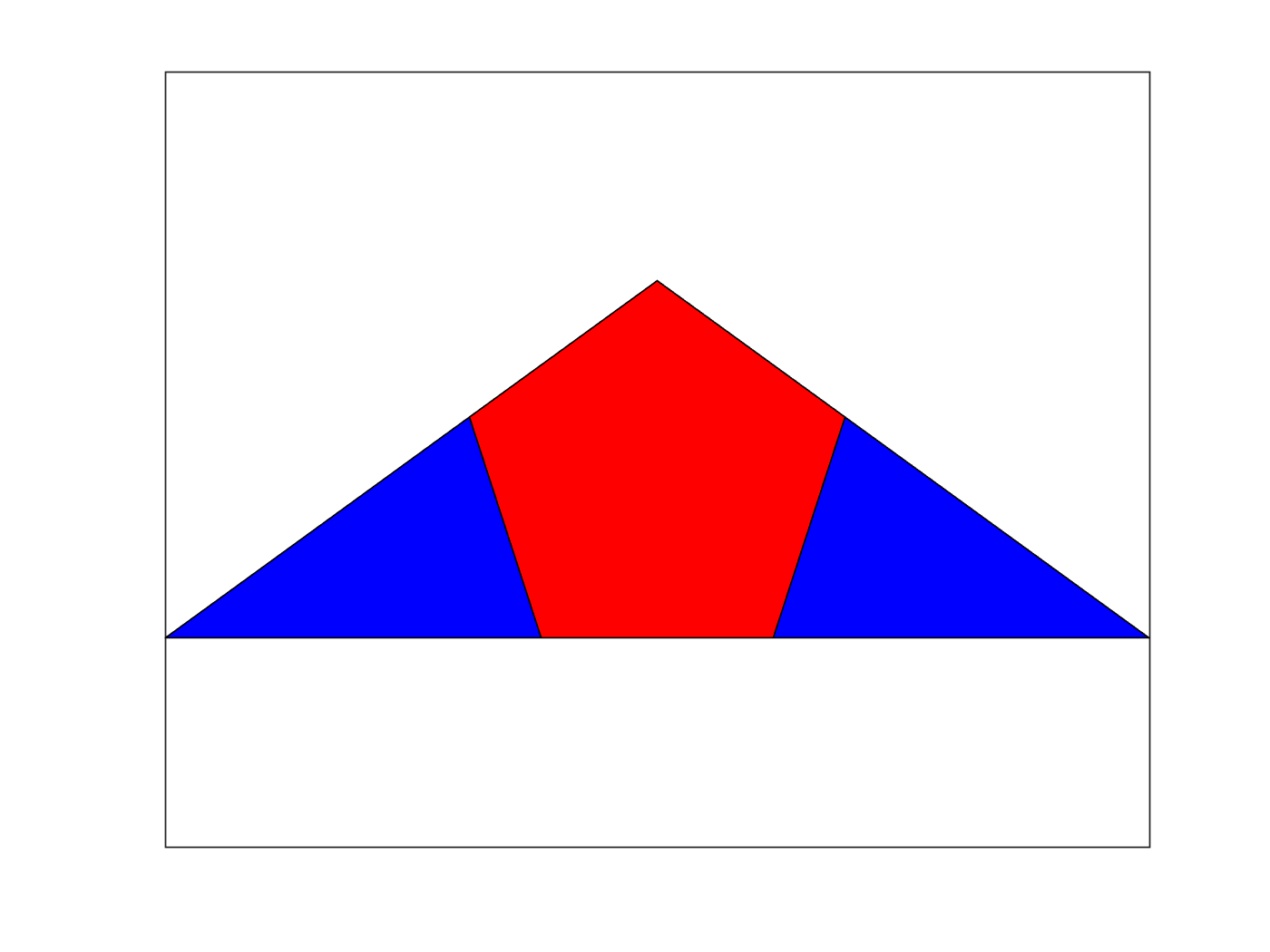}
\caption{Basic subdivisions for  the balanced tiling in Figure~\ref{fig:a}.}
\label{fig:basic}
\end{figure}

Basic subdivisions occur as patches in a balanced tiling $T(\ot, \mathcal S)$ when, in the
 tree $H(S_k)$ (as defined in Proposition~\ref{prop:h}), there is a vertex $i$ all of whose 
children are leaves.  In this case the basic subdivision is $B_i = \{ t(\ot, \mathcal S, k, \sigma) \,: \,
\sigma \in P_i\}$, where $P_i$ is the set of paths in the pre-tree $S_k$ corresponding to the
paths in $H(S_k)$ from the root to a child of vertex $i$.  
The union of the tiles in $B_i$ is
a set that is similar to the component $A_{h(i)}$ of the attractor, where $h$ is the homomorphism
of Proposition~\ref{prop:h}.

\begin{proposition}
\label{prop:hier} Given a parameter $\overleftarrow{\theta}\in
\overleftarrow{\Sigma}^{\infty}$ for a similarity GIFS, the sequences
\[
\mathcal{T }:= (T_{0}, T_{1}, T_{2}, \dots) \qquad\text{and} \qquad
\widehat{\mathcal{T}} := (\widehat{T}_{0}, \widehat{T}_{1}, \widehat{T}_{2},
\dots),
\]
where $T_{n} = T(\overleftarrow{\theta}, \mathcal{S}_{n})$ and $\widehat{T}%
_{n} = \widehat{T}(\overleftarrow{\theta}, \widehat{\mathcal{S}}_{n})$, have
the following properties:

\begin{enumerate}
\item $T_{0} = \widehat{T}_{0} = T({\overleftarrow{\theta}}, \mathcal{S})$;

\item $\mathcal{T}$ and $\widehat{ \mathcal{T}}$ are hierarchies for the
balanced GIFS tiling $T({\overleftarrow{\theta}}, \mathcal{S})$;

\item $\widehat{\mathcal{T}}$ is a \textit{refinement} of the hierarchy
$\mathcal{T}$ in that every tiling in $\mathcal{T}$ appears in
$\widehat{\mathcal{T}}$;

\item each tile in $\widehat{T}_{n+1}$ is either a tile in $\widehat{T}_{n}$
or the union of tiles in $\widehat{T}_{n}$ that comprise a basic subdivision.
\end{enumerate}
\end{proposition}

\begin{proof}
Statement (1) follows from the fact that $S_{(0,k)} = \widehat{S}_{(0,k)}$.
Statement (2) follows from Lemma~\ref{lem:hier}. Statement (3) follows
directly from the definitions.

Concerning statement (4), if tile $t = t({\overleftarrow{\theta}%
},\widehat{\mathcal{S}}_{n+1}, k, \sigma) \in\widehat{T}_{n+1}$ is not a tile
in $\widehat{T}_{n}$, then $0 < d (\sigma)-d (\overleftarrow{\theta}|k) +(n+1)
\leq d (\sigma_{*})$, but $0 \geq d(\sigma)-d (\overleftarrow{\theta}|k) + n$.
Therefore $d(\sigma)-d (\overleftarrow{\theta}|k) + n = 0$, which implies that
$\sigma\, e \in\widehat{S}_{(n,k)}$ for all $e \in E_{\sigma^{+}}$. Hence
$t({\overleftarrow{\theta}},\widehat{\mathcal{S}}_{n},k, \sigma\, e)
\in\widehat{T}_{n}$ for all $e \in E_{\sigma^{+}}$ and 
\[t = \bigcup_{e \in E_{\sigma^+}} \, t({\overleftarrow{\theta}},\widehat{\mathcal{S}}_n,k, \sigma\, e) = f_{\ot|k}\circ f_{\sigma}  \circ f_{e}(A_{e+}),\]
 which, by definition, is a basic subdivision.
\end{proof}

\subsection{The Tiling Map}

\label{sec:PT}

Given a similarity GIFS, denote by $\mathbb{T}$ the set of all balanced GIFS
tilings. Define a metric on $\mathbb{T}$ as follows. Let $\rho: {\mathbb{R}%
}^{d} \rightarrow\mathbb{S}^{d}$ be the usual stereographic projection onto
the $d$-sphere, obtained by positioning $\mathbb{S}^{d}$ at the origin. Let
$\mathbb{H}(\mathbb{S}^{d})$ be the set of nonempty closed subsets of
$\mathbb{S}^{d}$ and $d_{\mathbb{H}(\mathbb{S}^{d})}$ the corresponding
Hausdorff distance with respects to the round metric on $\mathbb{S}^{d}$. Let
$\mathbb{H}(\mathbb{H}(\mathbb{S}^{d}))$ be the collection of nonempty
compact subsets of the compact metric space $(\mathbb{H}(\mathbb{S}^{d}),
d_{\mathbb{H}(\mathbb{S}^{d})})$. For tilings $T, T^{\prime}$ define
\[
d^{\prime}(T,T^{\prime}) = d_{\mathbb{H }(\mathbb{H}(\mathbb{S}^{d}))}%
(\rho(T),\rho(T^{\prime})).
\]

Define the map $W:{\mathbb{T}}\rightarrow{\mathbb{T}}$ on the tiling space by
\[
W(T(\overleftarrow{\theta}))=f_{\theta_{1}}(T_{1}(\overleftarrow{\theta})),
\]
where $T_{1}$ is the level $1$ tiling in the hierarchy $\mathcal{T}$. In other
words, $W(T(\overleftarrow{\theta}))$ is, up to the similarity transformation
$f_{\theta_{1}}$, the tiling at level $1$ in the hierarchy of the tiling
$T(\overleftarrow{\theta},\mathcal{S})$. 

For any $\overleftarrow{\theta}= \overleftarrow{\theta}_{1} \,
\overleftarrow{\theta}_{2} \, \overleftarrow{\theta}_{3} \cdots\in
\overleftarrow{\Sigma}^{\infty}$, define the \textbf{shift map} $w :
\overleftarrow{\Sigma}^{\infty} \rightarrow\overleftarrow{\Sigma}^{\infty}$ on
the parameter space by $w(\overleftarrow{\theta}) = \overleftarrow{\theta}_{2}
\, \overleftarrow{\theta}_{3} \cdots$.

Call the map $T:\overleftarrow{\Sigma}^{\infty}\rightarrow\mathbb{T}$ from the
parameter space to the tiling space that takes $\overleftarrow{\theta}$ to
$T(\overleftarrow{\theta})$ the \textbf{tiling map}. The theorem below states
that, via the tiling map, the shift map on the paramater space corresponds to
the map $W$ on the tiling space.

\begin{theorem}
\label{thm:comm} The tiling map $T$ is continuous and the following diagram
commutes.
\[%
\begin{array}
[c]{ccccc}
&  & w &  & \\
& \overleftarrow{\Sigma}^{\infty} & \rightarrow & \overleftarrow{\Sigma
}^{\infty} & \\
T & \Big \downarrow &  & \Big \downarrow & T\\
& {\mathbb{T}} & \rightarrow & {\mathbb{T}} & \\
&  & W &  &
\end{array}
\]

\end{theorem}

\begin{proof}
Given an positive real number $R$ and two parameters ${\overleftarrow{\theta}%
}$ and $\overleftarrow{\theta^{\prime}}$, there is a $k$ such that both
patches $T({\overleftarrow{\theta}},k)$ and $T(\overleftarrow{\theta^{\prime}%
},k)$ contain the ball of radius $R$ centered at the origin. If
${\overleftarrow{\theta}}|k=\overleftarrow{\theta^{\prime}}|k$, i.e., if
${\overleftarrow{\theta}}$ and $\overleftarrow{\theta^{\prime}}$ are
sufficiently close in the parameter space metric, then
$T({\overleftarrow{\theta}},k)=T(\overleftarrow{\theta^{\prime}},k)$, which
guarantees that the $T({\overleftarrow{\theta}})$ and $T(\overleftarrow{\theta
^{\prime}})$ are close in the tiling space metric.

Concerning the commuting diagram we have
\[
\begin{aligned}
{f_{\theta_1}}^{-1} (T(w(\oo\theta)) &= \Big \{ f_{\oo\theta|k}\circ f_{\sigma} : \sigma = \sigma_1 \cdots \sigma_j, \;
0 <  d(\sigma) -  d(\theta|k) +  d(\theta_1) \leq d(\sigma_j) \Big \} \\ &=T_1(\oo\theta).\end{aligned}
\]
Therefore $T(w(\overleftarrow{\theta})) = W(T(\overleftarrow{\theta}))$.
\end{proof}

\section{Properties of Balanced Tilings}

\label{sec:properties}

\subsection{When is a balanced tiling self-similar?}

\begin{definition}
For a closed path $\alpha$ in a directed graph, let $\overline{\alpha} :=
\alpha\alpha\cdots$. A path $\theta\in\Sigma^{\infty}$ is \textit{periodic} if
there is an $\alpha\in\Sigma^{*}$ such that $\theta= \overline{\alpha}$ and
\textit{eventually periodic} if there are $\alpha, \beta\in\Sigma^{*}$ such
that $\theta=\beta\overline{\alpha}$.
\end{definition}

\begin{theorem}
\label{thm:SS} If $\overleftarrow{\theta}\in\overleftarrow{\Sigma}^{\infty}$
is eventually periodic, then the balanced GIFS tiling $T(\overleftarrow{\theta
})$ is self-similar.
\end{theorem}

\begin{proof}
Let ${\overleftarrow{\theta}} = \beta\, \overline\alpha$, where $\alpha,
\beta\in\overleftarrow{\Sigma}^{*}$. We may assume, without loss of
generality, that $k$ is sufficiently large so that ${\overleftarrow{\theta}}|k
= \beta\,\alpha^{q} \,\gamma$ for some $q \in\mathbb{N}$ and some $\gamma
\in\overleftarrow{\Sigma}^{*}$. Since $\alpha$ is a closed path, the end
vertex of $\beta\,\alpha^{q} \,\gamma$ is independent of the integer $q$. Let
$t = f_{{\overleftarrow{\theta}}|k}\circ f_{\sigma}(A_{\sigma^{+}}) \in
T({\overleftarrow{\theta}})$, where $0 < d(\sigma) - d( \beta\,\alpha^{q}
\,\gamma) \leq d(\sigma_{*}).$ It is sufficient to find a similarity
transformation $\phi$ such that $\phi(t)$ is tiled by tiles in
$T({\overleftarrow{\theta}})$.

As in the proof of Proposition~\ref{prop:balanced}, the set of paths
\[
S = \{ \omega\in\Sigma^{} \, : \, \omega^{-} = \sigma^{+} \; \text{and} \; 0 <
d(\omega) - d(\alpha) + d(\sigma) - d(\beta\,\alpha^{q} \,\gamma) \leq
d(\omega_{*}) \}
\]
is a pre-tree. Therefore, by Lemma~\ref{lem:overlap},
\[
A_{\sigma^{+}} = \bigcup\Big \{ f_{\omega}(A_{\omega^{+}}) : 0 <
d(\sigma\omega) - d(\beta\alpha^{q+1}\gamma) \leq d (\omega_{*})\Big \}.
\]
Let $\phi= f_{\beta}\circ f_{\alpha^{m}}\circ{f_{\beta}}^{-1}$, where $m$ is
the least positive integer such that $d(\beta\alpha^{m} \gamma) \geq
d(\sigma)$. Then $S \neq\emptyset$, and we have
\[
\begin{aligned} \phi(t) &= f_{\beta}\circ f_{\alpha}\circ {f_{\beta}}^{-1}
\circ f_{\ot|k}\circ f_{\sigma}(A_{\sigma^+}) =
f_{\beta}\circ  (f_{\alpha})^{q +1} \circ f_{\gamma} \circ f_{\sigma}(A_{\sigma^+})
\\ &= \bigcup \Big \{ (f_{\beta}\circ  (f_{\alpha})^{q +1} \circ f_{\gamma}) \circ f_{\sigma} \circ f_{\omega}(A_{\omega^+})
:  0 <  d(\sigma \omega) - d(\beta \alpha^{q+1}\gamma) \leq d (\omega_*)\Big \}
\\ &= \bigcup \Big \{ (f_{\beta}\circ  (f_{\alpha})^{q +1} \circ f_{\gamma}) \circ f_{\sigma\omega}
(A_{(\sigma \omega)^+})
: 0 < d(\sigma \omega)  - d(\beta \alpha^{q+1} \gamma)  \leq d ((\sigma \omega)_*)\Big \}.\end{aligned}
\]
But, by definition, the sets in the above union are tiles in
$T({\overleftarrow{\theta}})$.
\end{proof}

\subsection{When is a balanced tiling repetitive?}

 \begin{definition} \label{def:t}  Let $(G, F)$ be a GIFS with directed graph $G=(V,E)$.
We say that $(G,F)$ satisfies the {\bf triangle property} if there exists 
an integer $N$ such that, for all $u,v,w\in V$,
if $\sigma(u,v)$ is a path from $u$ to $v$ and $\sigma(u,w)$ is a path from $u$ to $w$ such that 
$d(\sigma(u,v)) - d(\sigma(u,w)) \geq N$, then there is a path $\sigma(w,v)$ from $w$ to $v$ such that
\[d(\sigma(u,v)) = d(\sigma(u,w)) + d(\sigma(w,v)) .\]

We say that $(G,F)$ is {\bf coprime} if there is a vertex $x$ in $G$ and closed paths $\omega_1, \omega_2, \cdots, \omega_m, \, m \geq 2,$ containing $x$ such that 
$\gcd \, (d(\omega_1), d(\omega_2), \cdots, d(\omega_m) )= 1$.
\end{definition}
 
\begin{lemma} \label{lem:tc}  If $(G,F)$ is coprime, then $(G,F)$ satisfies the triangle property.  
\end{lemma}

\begin{proof} Assume $(G,F)$ is coprime.  By a standard number theoretic result, there is an integer $N_1$ such that every integer greater than or equal to $N_1$ can
be expressed as $\sum_{i=1}^m a_i d(\omega_i)$ for some non-negative integers $a_i$.     Since $G$ is strongly connected, every two vertices are connected by a finite path.  Let $N_2$ be
an integer such that every two vertices are joined by a path whose $d$-value is at most $N_2$.  To show that $(G,F)$ satisfies the triangle property,  consider any three vertices
 $u,v,w\in V$, and let $N = N_1 + 2 N_2$.  Assume that $N_0 := d(\sigma(u,v)) - d(\sigma(u,w)) \geq N$. The path  $\sigma(w,v)$ required in the definition of the triangle property
is obtained as follows.   If $x$ is the vertex in the definition of coprime GIFS, let  $\sigma(w,v)$ be a path from $w$ to $v$ that is the concatenation of the three paths $\sigma(w,x)$,
$\sigma(x,x)$, and $\sigma(x,v)$, where $\sigma(w,x)$ and $\sigma(x,v)$ are any paths from $w$ to $x$ and from $x$ to $v$, respectively, such that $q_1 := d(\sigma(w,x))  \leq
 N_2, \, q_2 := d(\sigma(x,v))  \leq N_2$.  Let $\sigma(x,x)$ be obtained by, starting and ending at $x$, traversing the closed paths $\omega_1, \omega_2, \cdots, \omega_m$ 
(as defined in the definition of coprime) sufficiently many time so that $d(\sigma(x,x)) = N - q_1 - q_2$.  This is possible because   $N - q_1 - q_2\geq  N - 2 N_2 \geq N_1$.
Now $d(w,v) = q_1 +(N-q_1-q_2) + q_2 = N$.
\end{proof}

\begin{remark}If the digraph has just one vertex (the IFS case), then the GIFS is coprime if and only if $g := \gcd \{d(e), e\in E\}  = 1$, which can always, without loss of generality,  be assumed
by replacing the scaling constant $s$ by $s^{g}$.
\end{remark}

\begin{theorem}
\label{thm:r} If a similarity GIFS satisfies the triangle property, then the
balanced tiling $T(\overleftarrow{\theta})$ is repetitive for all parameters
$\theta$.
\end{theorem}

\begin{proof}
Since, by definition, $\theta$ is assumed filling for a balanced GIFS tiling,
any patch of $T(\overleftarrow{\theta})$ is contained in
$T(\overleftarrow{\theta}, k)$ for some $k$. It then suffices to show that
there exists an integer $n$ such that a copy congruent to
$T(\overleftarrow{\theta}, k)$ is contained in every tile of $T_{n}$ in the
$n^{th}$-level of the hierarchy $\mathcal{T}$ defined in
Section~\ref{sec:hier}. Let $v = \theta_{k}^{-}$, i.e., the last vertex of
$\overleftarrow{\theta}|k$. Choose $K>k$ and choose level $n$ such that the
following holds for every $t = f_{\overleftarrow{\theta}|K}\circ f_{\sigma
}(A_{\sigma^{+}}) \in T_{n}$: if $u = \theta_{K}^{-}, \; x = \sigma^{+}$, and
$\sigma(u,v) = \theta_{K} \theta_{K-1} \cdots\theta_{k+1}$, then $d
(\sigma(u,v)) - d(\sigma) \geq N$, where $N$ is as in Definition~\ref{def:t}.
By the triangle property, there is a path $\sigma(x,v)$ from $x$ to $v$ such
that $d(\sigma(u,v)) = d(\sigma(u,x)) + d(\sigma(x,v)).$ Now
\[
T({\overleftarrow{\theta}},k) = \{ f_{{\overleftarrow{\theta}}|k} \circ
f_{\omega} (A_{\omega^{+}}) \, : \, \omega^{-} = u, \; 0 < d(\omega) -
d({\overleftarrow{\theta}}|k) \leq d(\omega_{*})\}.
\]
If
\[
T = \{ f_{{\overleftarrow{\theta}}|K} \circ f_{\sigma(u,x) \sigma(x,v) \omega}
(A_{\omega^{+}}) \, : \, \omega^{-} = u,\; 0 < d(\omega) -
d({\overleftarrow{\theta}}|k) \leq d(\omega_{*})\},
\]
then $T({\overleftarrow{\theta}},k)$ and $T$ are isometric because
\[
\begin{aligned} d(\ot|k \,\omega) &= - d(\ot|k) + d(\omega) \\
&=- d(\ot|k) - d(\sigma(u,v) )+d(\sigma(u,x))
+d(\sigma(x,v) )+ d(\omega) \\&=  - d(\ot|K) +d(\sigma(u,x) \sigma(x,v) \omega).\end{aligned}
\]
But each tile in $T$ is contained in $t \in T_{n}$ because $f_{\sigma(x,v)
\omega} (A_{\omega^{+}}) \subset A_{\sigma^{+}}$ implies that
$f_{{\overleftarrow{\theta}}|K} \circ f_{\sigma(u,x)} \circ f_{\sigma(x,v)
\omega} (A_{\omega^{+}}) \subset f_{{\overleftarrow{\theta}}|K} \circ
f_{\sigma}(A_{\sigma^{+}})$.
\end{proof}

The following corollary follows immediately from Theorem~\ref{thm:r} and Lemma~\ref{lem:tc}.

\begin{cor}  If a similarity GIFS is coprime, then the balanced tiling $T(\oo \theta)$ is repetitive
for all parameters $\theta$. 
\end{cor}

\subsection{When is a balanced tiling non-periodic?}

\begin{definition}
\label{def:rigid} For a basic subdivison $B$, let $\cup B$ denote the union of the tiles in $B$.  A similarity GIFS is called \textbf{rigid} if the following
property holds for any basic $i$-subdivision $B$ and any $i'$-subdivision $B'$: if the intersection $B \cap B'$ is a nonempty set that tiles $(\cup B) \cap (\cup B')$, then $i=i'$ and $B =B'$.
\end{definition}

In Example~\ref{ex:squares} of Section~\ref{sec:examples}, there is a basic
subdivision $B$ into four unit squares. Clearly two congruent copies (red and
blue) of $B$ can intersect in a single square as in Figure~\ref{fig:square},
implying that the GIFS of Example~\ref{ex:squares} is not rigid.

\begin{figure}[htb]
\vskip -6mm
\includegraphics[width=5cm, keepaspectratio]{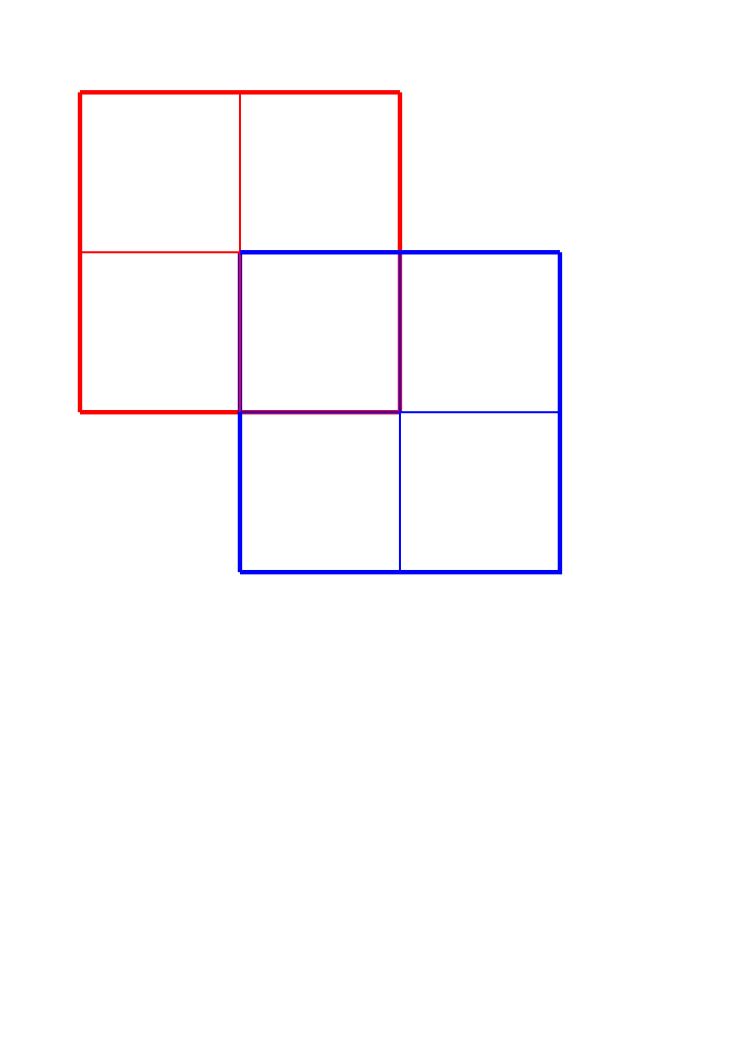}
\vskip -3cm
\caption{Showing non-rigidity of the GIFS in Example~\ref{ex:squares}.}
\label{fig:square}
\end{figure}

The following lemma states that, for a given rigid similarity GIFS, a tile preserving isometry
taking one balanced tiling onto another preserves the hierarchical structure.

\begin{lemma}
\label{lem:hier2} 
 If $(G,F)$ is rigid and if $\Phi$ is an isometry of $\R^d$ that maps $T(\oo\theta)$ to $T(\oo{\theta'})$, then $\Phi$
maps tiles in $\cup_{n=1}^{\infty} \widehat T_n(\oo \theta)$ to tiles in 
$\cup_{n=1}^{\infty} \widehat T_n(\oo{\theta '})$.
\end{lemma}

\begin{proof} 
 It will be proved by induction on the level $n$, that every
tile in $\widehat T_n(\oo\theta)$ is mapped by $\Phi$ to a tile in $\cup_{k=1}^{\infty} \widehat T_k(\oo{\theta'})$.  By the hypothesis, this is true for
$n=0$; assume that it is true for all $m < n$.  Assume, by way of contradiction, that a tile $t \in \widehat T_n(\oo\theta)$ is not mapped to a tile in  $\cup_{k=1}^{\infty} \widehat T_k(\oo{\theta'})$. Let $m$ be the largest level less than $n$
such that $t$ is tiled by a set $Q$ of at least $2$ tiles in $\widehat T_m(\oo\theta)$.  By the induction hypothesis, the image under $\Phi$ of each tile in $Q$ is a tile in $\cup_{k=1}^{\infty} \widehat T_k(\oo{\theta'})$.  For each  tile 
$q \in Q$, let $M_q$ be the least integer for which there exists a $t'_q\in \widehat T_{M_q}(\oo{\theta'})$ such that 
$\Phi(q) \subsetneq t'_q $.   There is some $q_0  \in Q$ such that no $q \in Q$ exists such that
$t'_q \subsetneq t' := t'_{q_0}$.  Let $m'$ be the largest integer less then $M_{q_0}$ such that $t'_{q_0} \notin 
\widehat T_{m'}$.   Let $Q'$ denote the set of tiles of $T_{m'}$ that tile $t'$.  Then $\Phi(t) \cap t'$
is tiled by $\Phi(Q) \cap Q'$.  But $Q$ is a basic subdivision of $t$ and $Q'$ is a basic subdivision of $t'$.
Then $\Phi(t) \neq t'$ contradicts the rigidity assumption.
\end{proof}

\begin{theorem}
\label{thm:np} If a similarity GIFS is rigid, then, for all
${\overleftarrow{\theta}}$, the balanced tiling $T(\overleftarrow{\theta})$ is
not periodic.
\end{theorem}

\begin{proof}
Assume, by way of contradiction, that $T = T(\overleftarrow{\theta})$ is
periodic. Then it has a translation symmetry $\tau$. By Lemma~\ref{lem:hier2},
the translation $\tau$ maps, for all $n$, tiles in $\widehat{T}_{n}$ to tiles
in $\bigcup_{n=0}^{\infty} \widehat{T}_{n}$. This leads to a contradiction since the size of the
tiles in $\widehat{T}_{n}$ eventually dwarfs the translational distance of
$\tau$.
\end{proof}

\subsection{When are two balanced tilings congruent?} \label{sec:cong}

Two tilings are said to be \textit{congruent} if there is a tile preserving isometry
taking one tiling to the other.

\begin{lemma}
\label{lem:up} If balanced tilings $T({\overleftarrow{\theta}})$ and
$T(\overleftarrow{\theta^{\prime}})$ for a rigid GIFS are congruent, then
there exist integers $K$ and $K^{\prime}$ such that $w^{K}%
({\overleftarrow{\theta}}) = w^{K^{\prime}}(\overleftarrow{\theta^{\prime}})$,
where $w$ is the shift map on parameter space.
\end{lemma}

\begin{proof}
Given a rigid GIFS $(G,F)$ and a parameter $\theta\in\overleftarrow{\Sigma
}^{\infty}$, let $t_{0} \in T({\overleftarrow{\theta}}) = \widehat{T}_{0}$.
Define a sequence of tiles
\[
t_{0} \subsetneq t_{1} \subsetneq t_{2} \subsetneq t_{3} \subsetneq\cdots
\]
inductively as follows: $t_{n+1}$ is the minimal (with respect to containment)
tile properly containing $t_{n}$ that belongs to a tiling in the hierarchy
$\widehat{T}_{0}({\overleftarrow{\theta}}), \widehat{T}_{1}%
({\overleftarrow{\theta}}),\widehat{T}_{2}({\overleftarrow{\theta}}), \dots$.
For a given $k$, if $t_{0} = t(\theta, \mathcal{S}, k, \sigma) =
f_{\overleftarrow{\theta}|k}\circ f_{\sigma}(A_{\sigma^{+}})$, let the
vertices of $\sigma$ be $v_{0}, v_{1}, v_{2}, \dots, v_{m}$, going from the
end of $\sigma$ to its beginning (the root of the pre-tree). Then
\begin{equation}
\label{eq:embed1}t_{j} = f_{\overleftarrow{\theta}|k}\circ f_{\sigma
|m-j}(A_{v_{j}})
\end{equation}
for $0 \leq j \leq m$.

Now assume that there is an isometry $\Phi$ that takes tiling
$T({\overleftarrow{\theta}})$ to $T(\overleftarrow{\theta}^{\prime})$. Let
$t_{0} \in T({\overleftarrow{\theta}})$ and $t_{0}^{\prime}= \Phi(t_{0}) \in
T(\overleftarrow{\theta}^{\prime})$. Let
\[
t_{0} \subsetneq t_{1} \subsetneq t_{2}\subsetneq t_{3} \cdots\qquad\text{and}
\qquad t^{\prime}_{0} \subsetneq t^{\prime}_{1} \subsetneq t^{\prime}%
_{2}\subsetneq t^{\prime}_{3} \cdots
\]
be the respective hierarchical sequences of tiles defined in the previous
paragraph. With notation as above, the formula corresponding to
Equation~\eqref{eq:embed1} is, with sufficiently large $k$,
\begin{equation}
\label{eq:embed2}t_{j} ^{\prime}= f_{\overleftarrow{\theta^{\prime}}|k}\circ
f_{\sigma^{\prime}|m^{\prime}-j}(A_{v^{\prime}_{j}})
\end{equation}
for $0 \leq j \leq m^{\prime}$. 

We now prove by induction that  $t^{\prime}_{n} = \Phi(t_{n})$ for all $n\geq0$.   It
is true by assumption for $n=0$.  Assume that  $t^{\prime}_{n} = \Phi(t_{n})$.
By Lemma~\ref{lem:hier2}, $\Phi(t_{n+1})$ is a tile in $\bigcup_{n=1}^{\infty} \widehat T_n$. 
If $\Phi(t_{n+1}) = t_{n+1}'$, then we are done; otherwise there is a tile $t' \in \bigcup_{n=1}^{\infty} \widehat T_n$
such that $t_n' \subsetneq t' \subsetneq \Phi(t_{n+1})$.  But, again by Lemma~\ref{lem:hier2}, $\Phi^{-1} (t')$
is a tile in $\bigcup_{n=1}^{\infty} \widehat T_n$ such that $t_n \subsetneq \Phi^{-1} \subsetneq t_{n+1}$,
which contradicts the definition of $t_{n+1}$ as the minimal tile properly containing $t_{n}$.

Since the basic subdivision of $t_n$ is the same as the basic subdivision
of $t_n'$ for all $n\geq 1$, by rigidity we have $v_0' = v_0,  v_1' = v_1, v_2' = v_2, \dots$.
Assuming, without loss of generality, that path
$\sigma$ is at least as long as $\sigma^{\prime}$, this implies that the part
of $\sigma^{\prime}$ between the vertex $v^{\prime}_{m^{\prime}}$ to vertex
$v^{\prime}_{0}$ (the whole path $\sigma^{\prime}$) is the same path as the
part of $\sigma$ from vertex $v_{m^{\prime}}$ to vertex $v_{0}$.

Now take $M$ sufficiently large so that the following holds. Expressing $t_{0} = f_{\overleftarrow{\theta
}|M}\circ f_{\widehat{\sigma}}(A_{\widehat{\sigma}^{+}})$, it must be the case
that $\widehat{\sigma} = \theta_{M} \, \theta_{M-1} \cdots\theta_{k+1} \sigma
$. Similarly $\widehat{\sigma} ^{\prime}= \theta_{M} \, \theta_{M-1}
\cdots\theta_{k+1} \sigma^{\prime}$. But by the last sentence of the previous
paragraph this implies that
\[
-\theta_{k+1} \, -\theta_{k+2} \cdots-\theta_{M-(m-m^{\prime})} =
-\theta^{\prime}_{k+1 +(m-m^{\prime})} \, -\theta_{k+2+(m-m^{\prime})}
\cdots-\theta^{\prime}_{M} .
\]
Since $M$ can be arbitrarily large, applying the shift operator yields
\[
w^{k} ({\overleftarrow{\theta}}) = w^{k+(m-m^{\prime})} (\overleftarrow{\theta
^{\prime}}).
\]

\end{proof}

\begin{definition}
\label{def:e} For a given GIFS, call two parameters ${\overleftarrow{\theta}},
\overleftarrow{\theta^{\prime}}$ \textbf{equivalent} if there exist integers
$K, K^{\prime}$ such that $d ({\overleftarrow{\theta}}|K) = d
(\overleftarrow{\theta^{\prime}}|K^{\prime})$ and $w^{K}%
({\overleftarrow{\theta}}) = w^{K^{\prime}}(\overleftarrow{\theta^{\prime}})$,
where $w$ is the shift map.  In other words, beginning segments of $\ot$ have the
same $d$-value, and the tails are identical.  
\end{definition}

\begin{theorem}
\label{thm:main} For a given GIFS and parameters ${\overleftarrow{\theta}}$
and $\overleftarrow{\theta^{\prime}}$ the balanced tilings
$T({\overleftarrow{\theta}})$ and $T(\overleftarrow{\theta}^{\prime})$ are
congruent if and only if ${\overleftarrow{\theta}}$ and $\overleftarrow{\theta
^{\prime}}$ are equivalent. Moreover $\Phi=f_{\overleftarrow{\theta^{\prime}%
}|K^{\prime}} \circ(f_{{\overleftarrow{\theta}}|K})^{-1}$ is an isometry
taking $T({\overleftarrow{\theta}})$ to $T(\overleftarrow{\theta}^{\prime})$ .
\end{theorem}

\begin{proof}
First assume that ${\overleftarrow{\theta}}$ and $\overleftarrow{\theta
^{\prime}}$ are equivalent. First note that $\Phi$ is an isometry because $d
({\overleftarrow{\theta}}|K) = d (\overleftarrow{\theta^{\prime}}|K^{\prime}%
)$. Suppose that $K, K^{\prime}$ are such that $d ({\overleftarrow{\theta}}|K)
= d (\overleftarrow{\theta^{\prime}}|K^{\prime})$ and $w^{K}%
({\overleftarrow{\theta}}) = w^{K^{\prime}}(\overleftarrow{\theta^{\prime}})$.
Assume, without loss of generality, that $k > \max(K,K^{\prime})$, and let $t
= f_{{\overleftarrow{\theta}}|k} \circ f_{\sigma} (A_{\sigma^{+}})$ be an
arbitrary tile in $T({\overleftarrow{\theta}})$. Then
\[
\begin{aligned}  \Phi(t) &= f_{\oo{\theta '}|K'} \circ ( f_{\ot|K}) ^{-1}\circ f_{\ot|k} \circ f_{\sigma} (A_{\sigma^+}) \\
&= f_{\oo{\theta '}|K'} \circ (f_{\theta_{K+1}})^{-1} \circ (f_{\theta_{K+2}})^{-1} \circ \cdots \circ
(f_{\theta_{k}})^{-1} \circ f_{\sigma} (A_{\sigma^+})  \\
&=  f_{\oo{\theta '}|K'} \circ (f_{\theta_{K'+1}})^{-1} \circ (f_{\theta_{K'+2}})^{-1} \circ \cdots \circ
(f_{\theta_{k+(K'-K)}})^{-1} \circ f_{\sigma} (A_{\sigma^+}) \\
&=  f_{\oo{\theta '}| (k+K'-K)} \circ f_{\sigma} (A_{\sigma^+}),\end{aligned}
\]
which is a tile in $T(\overleftarrow{\theta}^{\prime})$ because
\[
\begin{aligned} d( f_{\oo{\theta '}| (K'+k-K)} ) &= d( f_{\oo{\theta '}| K'}) + d(\ot|k) - d(\ot|K) =
d( f_{\oo{\theta }| K}) + d(\ot|k) - d(\ot|K) \\
&= d(\ot|k). \end{aligned}
\]

In the other direction, assume that there exists an isometry $\Phi$ that takes
$T({\overleftarrow{\theta}})$ to $T(\overleftarrow{\theta}^{\prime})$. By
Lemma~\ref{lem:up}, we may assume that ${\overleftarrow{\theta}} =
\overleftarrow{\alpha} \overleftarrow{\phi}$ and $\overleftarrow{\theta
^{\prime}} = \overleftarrow{\beta} \overleftarrow{\phi}$, where
$\overleftarrow{\alpha}, \overleftarrow{\beta} \in\Sigma^{*}$ and
$\overleftarrow{\phi}\in\Sigma^{\infty}$. If $d(\alpha) = d(\beta)$, then the
proof is complete. Otherwise, assume that $d(\alpha) < d(\beta)$. Let
$k(\alpha) := |\alpha|$ and $k(\beta):= |\beta|$. Let $t \in
T({\overleftarrow{\theta}}) $ and $t^{\prime}= \Phi(T({\overleftarrow{\theta}%
}))$. Choose $k > \max\, (k(\alpha), k(\beta))$ and such that $t\in
T({\overleftarrow{\theta}},k(\alpha) +k)$ and $t^{\prime}\in
T(\overleftarrow{\theta^{\prime}}, k(\beta) +k)$. We have
\[
t = f_{\overleftarrow{\theta}|k(\alpha)+ k}\circ f_{\sigma}(A_{\sigma^{+}})
\qquad\text{and} \qquad t^{\prime}= f_{\overleftarrow{\theta^{\prime}}%
|k(\beta)+k}\circ f_{\sigma^{\prime}}(A_{(\sigma^{\prime+}}),
\]
The last vertex, call it $x$, of the path $\overleftarrow{\phi}|k$, is the
first vertex of both $\sigma$ and $\sigma^{\prime}$ and is the root of the
corresponding pre-trees.

In what follows we refer to the notation and results in the proof of
Lemma~\ref{lem:up}. Since $d(\alpha) < d(\beta)$. the paths can be expressed
as
\[
\begin{aligned} \sigma &= \sigma_1 \, \sigma_2 \, \cdots \sigma_m \\
\sigma ' &= \sigma^{\prime \prime} \, \sigma_1 \, \sigma_2 \, \cdots \sigma_m , \end{aligned}
\]
where $\sigma^{\prime\prime}$ must be a closed path starting and ending at
vertex $x$. Moreover, since $t$ and $t^{\prime}$ are congruent, it must be the
case that $d(\sigma^{\prime\prime}) = d(\beta) - d( \alpha)$. Now, with $j$ an
integer such that $j > k + |\sigma^{\prime\prime}|$ we have
\[
\begin{aligned} \sigma &= \phi_{j} \, \phi_{j-1} \, \cdots \phi_{k+1} \, \sigma \\
\sigma ' &= \phi_{j} \, \phi_{j-1} \, \cdots \phi_{k+1} \, \sigma^{\prime \prime} \, \sigma .\end{aligned}
\]
Since the tails of these paths must be the same, we have $\phi^{\prime\prime}
:= \phi_{q} \, \phi_{q-1} \, \cdots\phi_{k+1} = \sigma^{\prime\prime},$ where
$q = k+|\sigma^{\prime\prime}|$. Therefore
\[
\begin{aligned}
\ot &= \oo{\alpha} \, \oo{\phi} = \oo{\alpha} \, \oo{\phi|k} \, \oo{ \phi ^{\prime \prime} } \, w^q(\oo{\phi}) \\
\oo{\theta '} &= \oo{\beta} \, \oo{\phi} = \oo{\beta} \, \oo{\phi|k} \, \oo{ \phi ^{\prime \prime} } \, w^q(\oo{\phi} ).
\end{aligned}
\]
Recall that $\phi^{\prime\prime} = \sigma^{\prime\prime}$ is a closed path and
that $d(\phi^{\prime\prime}) = d(\sigma^{\prime\prime} ) = d(\beta) - d(
\alpha)$. Letting $K =|\alpha|+k+|\phi^{\prime\prime}|$ and $K^{\prime}=
|\beta| + k|$, this implies that
\[
w^{K}({\overleftarrow{\theta}}) = w^{|\alpha|+k+|\phi^{\prime\prime}|
}({\overleftarrow{\theta}}) = w^{|\beta| + k} (\overleftarrow{\theta^{\prime}%
}) = w^{K^{\prime}}(\overleftarrow{\theta^{\prime}})
\]
and
\[
\begin{aligned}
d( \ot|K) &= d(\alpha) + d(\phi|k) + d(\phi^{\prime \prime}) = d(\alpha) + d(\phi|k) + d(\beta)-d(\alpha) =
d( \beta) + d(\phi|k) \\ &= d( \oo{\theta '}|K'),\end{aligned}
\]
proving that ${\overleftarrow{\theta}}$ and $\overleftarrow{\theta^{\prime}}$
are equivalent.
\end{proof}

\subsection{When are there uncountably many balanced
tilings?}

\begin{theorem}
\label{thm:u} If a similarity GIFS $(G,F)$ is rigid, then there are an uncountable
number of balanced tilings for $(G,F)$ up to congruence.
\end{theorem}

\begin{proof}
Let ${\overleftarrow{\theta}}$ be a parameter of the GIFS. Define
${\overleftarrow{\theta}}$ to be \textit{disjunctive} if any finite path
$\omega\in\Sigma^{*}$ appears as a subpath of ${\overleftarrow{\theta}}$,
i.e., there are paths $\alpha\in\Sigma^{*}, \, \sigma\in\Sigma^{\infty}$, such
that ${\overleftarrow{\theta}} = \alpha\, \omega\, \sigma$. By a proof as in
\cite{BV1}, a disjunctive parameter is filling. The set $X$ of parameters
modulo equivalence is uncountable. By \cite{Sta}, the set $X$ contains
uncountably many disjunctive parameters. Therefore by Theorem~\ref{thm:main},
there are an uncountable number of parameters that are simultaneously
disjunctive and pairwise have the distinct tails property.
\end{proof}

\vskip 2mm

\section*{Acknowledgements} 
This work was partially supported by a grant from the Simons Foundation (\#322515 to Andrew Vince).

\end{document}